\documentclass[11pt, reqno]{article}

\usepackage[utf8]{inputenc}
\usepackage[T1]{fontenc}

\usepackage{setspace}
\usepackage[margin=1.25in]{geometry}
\usepackage{parskip}
\usepackage[sc]{mathpazo}
\usepackage[T1]{eulervm}
\usepackage[scaled]{helvet}

\usepackage{amsmath,amsthm, amssymb, amsfonts}
\usepackage{mathtools}
\usepackage{graphicx}
\usepackage[hidelinks]{hyperref}

\newcommand{\arxiv}[1]{\href{https://arxiv.org/abs/#1}{\texttt{ArXiv:#1}}}
\newcommand{\arxivmath}[1]{\href{https://arxiv.org/abs/math/#1}{\texttt{ArXiv:#1}}}
\newcommand{\arxivmaph}[1]{\href{https://arxiv.org/abs/math-ph/#1}{\texttt{ArXiv:#1}}}

\theoremstyle{plain}
\newtheorem{thm}{Theorem}

\newtheorem{lem}{Lemma}[section]
\newtheorem{prop}{Proposition}[section]

\theoremstyle{definition}
\newtheorem{defn}{Definition}
\newtheorem{rem}{Remark}[section]

\numberwithin{equation}{section}

\newcommand{\E}[1]{\mathbf{E}\left[#1\right]}
\newcommand{\pr}[1]{\mathbf{Pr}\left[#1\right]}
\newcommand{\ind}[1]{\mathbf{1}_{\{ #1 \}}}
\newcommand{\dt}[1]{\mathrm{det}\left(#1\right)}

\newcommand{\D}{\Delta}

\newcommand{\vt}{\, | \,}
\newcommand{\veps}{\vec{\varepsilon}}

\newcommand{\eps}{\varepsilon}
\newcommand{\Z}{\mathbb{Z}}
\newcommand{\R}{\mathbb{R}}
\newcommand{\C}{\mathbb{C}}

\newcommand{\Ai}{\mathrm{Ai}}

\newcommand{\hb}{\mathcal{H}}
\newcommand{\G}{\mathcal{G}}
\newcommand{\Gb}{\mathbold{G}}
\newcommand{\J}{\mathcal{J}}

\title{Revisiting the temporal law in KPZ random growth}
\author{Mustazee Rahman \thanks{\textsc{Department of Mathematical Sciences, Durham University}. \textit{Email}: \texttt{mustazee@gmail.com}}}
\date{}

\begin{document}

\maketitle

\begin{abstract}
    This article studies the temporal law of the KPZ fixed point. For the stationary geometry, we find the two-time law, which extends the single time law due to Baik-Rains and Ferrari-Spohn. For the droplet geometry, we find a relatively simpler formula for the multi-time law compared to a previous formula of Johansson and the author. These formulas are derived as the scaling limit of corresponding multi-time formulas for geometric last passage percolation.
\end{abstract}

\section{Introduction} \label{sec:intro}

\subsection{Local random growth}
We study a model of local random growth describing growing interfaces in the (1+1)-dimensional spacetime plane.
It forms a basic model of non-equilibrium statistical mechanics, exhibiting novel scaling exponents and fluctuations, and a conjectural promise of universality.
Part of what makes such models interesting is that they lie at the crosswords of many fields: algebraic combinatorics, asymptotic representation theory, differential equations, integrable probability, interacting particle systems, random geometry, random permutations, random matrices and random tilings.
See the surveys \cite{BoGo, CoKPZ, Josurvey, QuKPZ, Tak, Zyg} and references therein.

The model is geometric last passage percolation or discrete polynuclear growth (PNG).
Consider parameters $a_i, b_j \in [0,1]$ for $i,j \geq 1$ such that $0 < a_ib_j < 1$ for every $(i,j)$.
Decorate the first quadrant of $\Z^2$ with independent, random weights $\omega(i,j)$ having the law
$\mathrm{Geom}(a_i b_j)$, that is,
$$ \pr{\omega_{i,j} = k} = (1- a_ib_j) (a_i b_j)^k \quad k = 0,1,2,\ldots.$$
The growth function associated to this random environment is
\begin{equation} \label{eqn:G} \Gb(m,n) = \max \, \{ \Gb(m-1,n) , \Gb(m,n-1)\} + \omega(m,n).\end{equation}
The boundary conditions are zero: $\Gb(m,0) \equiv 0 \equiv \Gb(0,n)$.
This is a model of local random growth -- growth because $\Gb(m,n)$ is increasing in both components,
and local because the growth rule depends on the two neighbouring values at $(m-1,n)$ and $(m,n-1)$.

Unwrapping the recursive definition \eqref{eqn:G} leads to the expression
$$\Gb(m,n) = \max_{\pi} \, \sum_{(i,j) \in \pi} \omega(i,j)$$
where the maximum is over all up/right lattice paths $\pi$ from $(1,1)$ to $(m,n)$.
Up/right means that the paths move in the direction (0,1) or (1,0) at each step.
In this way $\Gb(m,n)$ represents the last passage time among directed paths from $(1,1)$ to $(m,n)$
with respect to the weights $\omega(i,j)$.
This perspective leads to the more general definition of point-to-point last passage values:
\begin{equation} \label{eqn:GmnMN} \Gb((m,n) \to (M,N)) = \max_{\pi} \, \sum_{(i,j) \in \pi} \omega(i,j)\end{equation}
where the maximization is over all up/right paths from $(m,n)$ to $(M,N)$.

One may also consider a point-to-line last passage problem.
The ``line" is introduced as a boundary condition for the process $\Gb$.
Namely, set $\Gb(0,n) = x_n$ for $n \geq 1$ where $x_1 \leq x_2 \leq x_3 \leq \cdots$ are integers.
Then define
\begin{equation} \label{eqn:Gx} \Gb(m,n \mid x) = \max_k x_k + \Gb((0,k) \to (m,n)).\end{equation}

Two scaling limits arise naturally from these definitions.
The first is the limit of the point-to-point process as a four-parameter family: $(m,n; M,N) \mapsto \Gb((n,m) \to (M,N))$.
The second is the limit of the point-to-line process for various choices of the line: $(m,n) \mapsto \Gb(m,n \mid x)$.
Both these scaling limits have been found, leading to a breakthrough in the study of random growth models and the universality of their fluctuations.
Universality here refers to the idea that scaling exponents and limiting fluctuations should not be model dependent.
PNG is part of a whole host of models that comprise the Kardar-Parisi-Zhang (KPZ) universality class, named for the seminal work of Kardar, Parisi and Zhang \cite{KPZ}.
The aforementioned scaling limits (which we will elaborate below) have a rich structure as spacetime objects, and the purpose of this article is to understand the temporal law of these limit objects.

\subsection{The directed landscape}
The scaling limit of the point-to-point last passage process, $\Gb((m,n) \to (M,N))$, is a random geometry called the directed landscape.
Introduced by Dauvergne, Ortmann and Vir\'{a}g \cite{DOV}, the directed landscape 
$$ \mathcal{L} : \R_{\uparrow}^4 = \{(x,s;y,t) \in \R^4: s < t\} \to \R$$
is a random continuous function such that $\mathcal{L}(x,s;y,t)$ is viewed as a distance from $(x,s)$ to $(y,t)$.
It admits geometric concepts such as length, geodesics, bisectors, Busemann functions, fractal dimensions, and more.
See the articles \cite{BB, BasGH, BatGH,Bh,BSS,DVS,Dau,GZ,RV1,RV2} and references therein.

From the perspective of random geometry, $\Gb((m,n) \to (M,N))$ is a kind of random metric.
The scaling under which $\Gb$ converges to $\mathcal{L}$ goes as follows \cite{DV}. Let $a_i = b_j = \sqrt{q}$ for every $i,j$. There are $q$-dependent constants $c_1, c_2, c_3$ such that as $T$ tends to infinity,
$$\frac{\Gb\left ((Ts - c_1 T^{2/3}x, Ts + c_1 T^{2/3}x) \to ( Tt - c_1 T^{2/3}y, Tt + c_1 T^{2/3}y)\right) - c_2T(t-s)}{c_3 T^{1/3}} \to \mathcal{L}(x,s;y,t)$$
in law, in the topology of uniform convergence over compact subsets of $\R_{\uparrow}^4$.

\subsection{The KPZ fixed point}
The scaling limit of the point-to-line last passage process is a scale invariant Markov process with state space consisting of upper semicontinuos functions $h : \R \to \R$. It describes a randomly growing interface
$$ \mathbold{h}(x,t; h_0) : \{(x,t) \in \R^2: t > 0\} \to \R$$
called the KPZ fixed point with initial condition $h_0$. It was introduced by Matetski, Quastel and Remenik \cite{MQR}. The process is Markovian in time, and its fixed time law enjoys a degree of integrability: the finite dimensional distributions of $x \mapsto \mathbold{h}(x,t;h_0)$ are given by a Fredholm determinant. See also \cite{QR, MQR2} for more on the integrability of these laws.

The relation between the KPZ fixed point and the directed landscape is the continuum analogue of \eqref{eqn:Gx} and akin to the Hopf-Lax formula for Hamilton-Jacobi equations:
\begin{equation} \label{eqn:var}
    \mathbold{h}(x,t;h_0) = \sup_{y \in \R}\, \{ h_0(y) + \mathcal{L}(y,0;x,t)\}.
\end{equation}
See \cite{NQR, DV, RV1} for proofs of this variational formula.

To emphasize the point-to-line nature of the KPZ fixed point, let us denote
\begin{equation} \label{eqn:L} \mathcal{L}(h_0; x,t) := \mathbold{h}(x,t;h_0).\end{equation}

\subsection{Temporal laws} \label{sec:temporal}
We study the temporal law of the KPZ fixed point for two particular initial conditions. They are the droplet or narrow wedge: $h_0(y) = -\infty \ind{y \neq 0}$, and the stationary or Brownian: $h_0(y) = B(y)$, a two-sided Brownian motion with diffusivity constant 2. For the droplet, this process is simply $(x,t) \mapsto \mathcal{L}(0,0;x,t)$. For the Brownian, its time evolution is stationary, meaning $\mathcal{L}(B; x,t) - \mathcal{L}(B; 0,t) \stackrel{d}{=} B(x)$ \cite{Pim}, but there is a non-trivial coupling in the process $t \mapsto \mathcal{L}(B;0,t)$. The KPZ fixed point even enjoys an ergodicity property: if $h_0$ has asymptotic drift zero, then $\mathcal{L}(h_0; (x,t)) - \mathcal{L}(h_0; (0,t))$ converges in law to $B$ in the large time limit (in fact, this convergence holds in a stronger sense; see Theorem 1 of \cite{Pim} and Theorem 2.1 of \cite{BSS}).

We shall study the so-called multi-time law of $\mathcal{L}(0,0;x,t)$ and the two-time law of $\mathcal{L}(B;(x,t))$.
Temporal laws are important because an early motivation for studying random growth models was to understand their time evolution.
We quote the authors from \cite{MQR}.

\begin{quote}
In modelling, for example, edges of bacterial colonies, forest fires, or spread of genes,
the non-linearities or noise are often not weak, and it is really the fixed point that should be used in
approximations and not the KPZ equation. However, progress has been hampered by a complete lack
of understanding of the time evolution of the fixed point itself. Essentially all one had was fixed time
distributions of a few special self-similar solutions, the Airy processes.
\begin{flushright}
    \footnotesize{Matetski, Quastel and Remenik -- in relation to the KPZ fixed point}
  \end{flushright}
\end{quote}

The temporal law of the KPZ fixed point has been studied recently. Works of Baik and Liu \cite{BL1,BL2} discovered multi-time formulas for periodic tasep (see also \cite{Liao}), which led to Liu \cite{Liu} deriving the multi-time distributions of $\mathcal{L}(0,0;x,t)$ and $\mathcal{L}(h_0\equiv 0; x,t)$. Johansson discovered the two-time distribution for both geometric PNG and Brownian last passage percolation \cite{JoTwo1, JoTwo2}, which led to formulas for the CDF of the pair $(\mathcal{L}(0,0;x_1,t_1), \mathcal{L}(0,0;x_2,t_2)$. Johansson and the author found a formula for the multi-time law of geometric PNG \cite{JR1}, which led to a formula for the multi-time law of $\mathcal{L}(0,0;x,t)$.

In this article we present a formula for the two-time law of the stationary KPZ fixed point $\mathcal{L}(B; x,t)$.
The single time law of $\mathcal{L}(B;x,t)$ is the Baik-Rains distribution, a formula for which is given in \cite{BR} as well as in \cite{FS}.
Furthermore, \cite{BFP} derives a formula for the spatial law $x \mapsto \mathcal{L}(B;x,t)$.
The other main result is a relatively simpler formula than in \cite{JR1} for the multi-time law of $\mathcal{L}(0,0;x,t)$.
It reduces the formula from \cite{JR1} that involves 7 families of integral kernels to 5 such families. The main idea is
to employ a different conjugation factor in the determinant formula, which makes the limit transition from geometric PNG to the KPZ fixed point simpler.

\subsection{Main results} \label{sec:mainresults}

\subsubsection{Two-time law for stationary random growth}
Let $B(y)$ be a two-sided Brownian motion with diffusivity constant 2 ($\E{B(y)^2} = 2|y|$).
Consider the KPZ fixed point started from $B$:
$$ \mathcal{L}(B; x,t) = \sup_{y \in \R} \{ B(y) + \mathcal{L}(0,y;x,t)\}.$$
One may also add drift to the Brownian motion, to get (by the affine invariance of the KPZ fixed point):
$$ \mathcal{L}(B+ay; x,t) \stackrel{d}{=} \mathcal{L}(B; x + \frac{a}{2}t, t ) + ax + \frac{a^2}{4}t.$$
We shall present a formula for the joint law of $\mathcal{L}(B; x_1,t_1)$ and $\mathcal{L}(B;x_2,t_2)$ for $t_1 < t_2$.

Define the function:
\begin{equation} \label{eqn:Gtxx2} \G(z \vt t, x, \xi) = \exp \left \{ \frac{t}{3}z^3 + t^{2/3}x w^2 - t^{1/3} \xi z \right \}.\end{equation}

Fix $0 < t_1 < t_2$, $x_1, x_2 \in \R$ and $\xi_1, \xi_2 \in \R$. Set
\begin{equation} \label{eqn:deltas}
\D t = t_2 - t_1, \quad \D x = (\frac{t_2}{\D t})^{2/3} x_2 - (\frac{t_1}{\D t})^{2/3}x_1, \quad \D \xi = (\frac{t_2}{\D t})^{1/3} \xi_2 - (\frac{t_1}{\D t})^{1/3}\xi_1.
\end{equation}
Let $\Gamma_d$ be a vertical contour, oriented upwards, with real part $d \in \R$.

\begin{defn} \label{def:A}
    Consider the following four integral kernels acting on $L^2(\R)$.
    \begin{align*}
        A_1(u,v) &= e^{\mu(v-u)} \ind{v \leq 0}\, \frac{1}{(2\pi \mathbold{i})^2}\oint \limits_{\Gamma_{-d_1}} d\zeta \oint \limits_{\Gamma_{D_1}} dz\,
        \frac{\G(z \vt t_1,x_1,\xi_1) e^{zv-\zeta u}}{\G(\zeta \vt t_1, x_1, \xi_1) (z-\zeta)}.\\
        A_2(u,v) &= e^{\mu(v-u)} \, \ind{u > 0}\,  \frac{1}{(2\pi \mathbold{i})^2}\oint \limits_{\Gamma_{-d_1}} d \zeta \oint \limits_{\Gamma_{D_1}} dz\, 
        \frac{\G(z \vt \D t, \D x, \D \xi) e^{\zeta v - zu}}{\G(\zeta \vt \D t, \D x, \D \xi) (z-\zeta)}. \\
        A_3(u,v) &= e^{\mu(v-u)} \, \frac{1}{(2\pi \mathbold{i})^4} \oint \limits_{\Gamma_{-d_1}} d\zeta \oint \limits_{\Gamma_{-d_2}} d\omega \oint \limits_{\Gamma_{D_1}} dz \oint \limits_{\Gamma_{D_2}} dw\, \\
        & \frac{\G(z \vt t_1,x_1,\xi_1) \G(w \vt \D t, \D x, \D \xi)e^{\omega v-\zeta u}}{\G(\zeta \vt t_1, x_1, \xi_1) \G(\omega \vt \D t, \D x, \D \xi)(z-\zeta)(w-\omega)(z-w)}.\\
        A_4(u,v) &= e^{\mu(v-u)} \, \frac{1}{(2\pi \mathbold{i})^4} \oint \limits_{\Gamma_{-d_1}} d\zeta \oint \limits_{\Gamma_{-d_2}} d\omega \oint \limits_{\Gamma_{D_1}} dz \oint \limits_{\Gamma_{D_2}} dw\, \\
        & \frac{\G(z \vt t_1,x_1,\xi_1) \G(w \vt \D t, \D x, \D \xi)e^{\omega v-\zeta u}}{\G(\zeta \vt t_1, x_1, \xi_1) \G(\omega \vt \D t, \D x, \D \xi)(z-\zeta)(w-\omega)(z-w)}.
    \end{align*}
    Here $d_1,d_2,D_1,D_2 > 0$. In the formula for $A_3$, $D_1 < D_2$, and for $A_4$, $D_1 > D_2$. The parameter $\mu$ is a sufficiently large constant in terms of $t_i,x_i$ (it suffices to have $\mu > \max \{x_1t_1^{-1/3}, \D x(\D t)^{-1/3}\}$).
\end{defn}
The kernels $A_i$ are all trace class for sufficiently large $\mu$ because each can be written as a product of two Hilbert-Schmidt kernels.

For $\theta \in \C \setminus \{0\}$, define the kernel
\begin{equation} \label{eqn:Atheta}
[A(\theta)](u,v) = \theta^{\ind{u > 0}}(A_2 - A_1 + A_3)(u,v) + \theta^{-\ind{u \leq 0}}(A_1 - A_2 + A_4)(u,v).
\end{equation}
Define the quantity
\begin{equation} \label{eqn:Ttheta}
    T(\theta) = t_1^{1/3} \xi_1 + \int_{-\infty}^{\infty}du \int_{-\infty}^0 dv \, \theta^{\ind{u > 0}}[(I + A(\theta))^{-1}\cdot A(\theta)](u,v) e^{\mu(v-u)}
\end{equation}
The quantity $\mu$ is from Definition \ref{def:A}.
We shall see that $T(\theta)$ is a well-defined and finite quantity for all $\theta$ outside a discrete subset of $\C$.
In particular, there is a choice of $r > 1$ such that $T(\theta)$ is well-defined and bounded over all $|\theta| = r$.

\begin{thm} \label{thm:stat}
    Let $0 < t_1 < t_2$, $x_1, x_2 \in \R$ and $\xi_1, \xi_2 \in \R$. Define the function
    $$ F_0(t_1,x_1,\xi_1; t_2, x_2,\xi_2) = \frac{1}{2 \pi \mathbold{i}} \oint \limits_{|\theta|=r}d\theta\, 
    \frac{(\partial_{\xi_1} + \partial_{\xi_2}) \dt{I + A(\theta)}_{L^2(\R)} \cdot T(\theta)}{\theta - 1}.$$
    Here $r > 1$ can be chosen such that all quantities are well-defined.

    Let $B$ be a Brownian motion with diffusivity constant 2. Then,
    \begin{equation} \label{eqn:Fstat}
    \Pr(\mathcal{L}(B; t_1,x_1) \leq \xi_1, \mathcal{L}(B; t_2,x_2) \leq \xi_2) = F_0(t_1,x_1,\xi_1; t_2, x_2, \xi_2).
    \end{equation}
\end{thm}

The kernel $A(\theta)$ is in fact the kernel $F(\theta)$ of Theorem \ref{thm:droplet} below in the case $p=2$.
It was introduced in \cite{JoTwo2} for the two-time distribution in KPZ random growth.

It would be interesting to study special cases of the two-time law $F_0$.
For instance, one could look at the short and long time separation limits ($t_2/t_1 \to 1$ and $t_2/t_1 \to \infty$).
For the droplet geometry, this was studied by Johansson in \cite{JoTwospecial}.
One can derive the single time law (Baik-Rains distribution) by taking the limit $\xi_2 \to +\infty$.
In this limit, the only kernel that survives is $A_1$, and one finds the Baik-Rains distribution in the form
derived by Ferrari-Spohn \cite{FS} (see also \cite{BFP}).

\subsubsection{Simplified multi-time law for the droplet}
To present this result we need to introduce notation, most of it directly from \cite{JR1}.

Consider times $0<t_1<t_2 < \cdots < t_p$, points $x_1,x_2, \ldots x_p \in\R$ and $\xi_1,\xi_2, \ldots, \xi_p \in \R$.

\paragraph{\textbf{Delta notation}}
For integers $0 \leq k_1 < k_2 \leq p$, define
\begin{align}\label{deltanotation}
\D_{k_1, k_2} t &= t_{k_2} - t_{k_1} \quad \text{and}\;\; \D_{k} t= t_k - t_{k-1}, \\ \nonumber
\D_{k_1, k_2} x & = x_{k_2} \, \Big (\frac{t_{k_2}}{\D_{k_1, k_2}t} \Big)^{\frac{2}{3}} - x_{k_1} \, \Big(\frac{t_{k_1}}{\D_{k_1, k_2}t} \Big)^{\frac{2}{3}}
\quad \text{and}\;\; \D_k x = \D_{k-1,k} x\, , \\ \nonumber
\D_{k_1, k_2} \xi & = \xi_{k_2} \, \Big (\frac{t_{k_2}}{\D_{k_1, k_2}t} \Big)^{\frac{1}{3}} - \xi_{k_1} \, \Big(\frac{t_{k_1}}{\D_{k_1, k_2}t} \Big)^{\frac{1}{3}}
\quad \text{and}\;\;  \D_k \xi = \D_{k-1,k} \xi \, .
\end{align}
By convention, $y_0 = 0$ for $y = t,x,\xi$. We will also use the shorthand
\begin{equation*}
\D_{k_1,k_2} (y^1, \ldots, y^{\ell}) = (\D_{k_1,k_2}y^1, \ldots, \D_{k_1,k_2}y^{\ell}) \quad \text{and}\;\;
\D_k (y^1, \ldots, y^{\ell}) = (\D_k y^1, \ldots, \D_k y^{\ell}).
\end{equation*}

\paragraph{\textbf{Theta factors}}
For $\vec{\eps} = (\eps_1, \ldots, \eps_{p-1}) \in \{1,2\}^{p-1}$ and
$\theta = (\theta_1, \ldots, \theta_{p-1}) \in (\C \setminus 0)^{p-1}$, define the following quantities.
\begin{align} \label{thetaeps}
	\theta(r \vt \veps) &= \prod_{k=1}^{r-1} \theta_k^{2-\eps_k} \, \prod_{k=r}^{p-1} \theta_k^{1-\eps_k} \quad \text{for}\;\; 1 \leq r \leq p.
\end{align}
Notable $\veps$ will be
$$ \eps^{k} = (\overbrace{2, \ldots,2}^{k-1}, 1, \ldots, 1) \quad \text{for}\;\; 1 \leq k \leq p.$$
For these, define
\begin{equation} \label{theta}
\Theta(r \vt k) = \theta(r \vt \eps^k) - (1- \ind{r=p, k=p-2}) \cdot \theta(r \vt \eps^{k+1}) \quad \text{for}\;\; 1 \leq k < \min \{r,p-1\}, \;\; 1 \leq r \leq p.
\end{equation}
Set $\Theta(r \vt k)$ to be zero otherwise. Set also
$$ (-1)^{\eps_{[k_1,k_2]}} = (-1)^{\sum_{k = \max \{1,k_1\}}^{\min \{k_2,p-1\}} \, \eps_k} \quad \text{for}\;\; 0 \leq k_1 < k_2 \leq p.$$
It will be convenient to write  $(-1)^{\eps_{[k_1,k_2]}} \cdot (-1)^x$ as $(-1)^{\eps_{[k_1, k_2]}+x}$.

\paragraph{\textbf{Contour notation}}
We will denote the contour integral
$$ \frac{1}{2 \pi \, \mathbold{i}} \int\limits_{\gamma} dz \quad  \text{as} \quad \oint\limits_{\gamma} dz\,.$$
There will be two types of contours in our calculations: circles and vertical lines.
Throughout, $\gamma_r$ denotes a circular contour around the origin of radius $r > 0$
with counterclockwise orientation. Also, $\gamma_r(1)$ is such a circular contour
around 1. A vertical contour through $d \in \R$ oriented upwards is denoted $\Gamma_d$.

For $p \geq 1$ consider the Hilbert space
$$ \hb = \underbrace{L^2(\R_{<0}) \oplus \cdots \oplus L^2(\R_{<0})}_{p-1} \oplus \, L^2(\R_{>0}).$$
A kernel $F$ on $\hb$ has a $p \times p$ block structure, and we denote by $F(r,u; s,v)$ its $(r,s)$-block. So
$$ F(u,v) = \begin{bsmallmatrix}
F(1,u;\, 1,v)& \cdots&  F(1,u;\,p,v)\\
\vdots &  & \vdots \\
F(p,u;\,1,v)& \cdots  & F(p,u;\,p,v)
\end{bsmallmatrix}_{p \times p}\,. $$

Define the function
\begin{equation} \label{Gtxx}
\G(w \vt t, x, \xi) = \exp \Big \{ \, \frac{t}{3} w^3 + t^{\frac{2}{3}} xw^2 - t^{\frac{1}{3}} \xi w \, \Big \}
\quad \text{for}\; w \in \C \;\text{and}\; t,x,\xi \in \R.
\end{equation}

Introduce the notation
\begin{align} \label{blocknot}
r^{*} &= \min \{r,p-1\} \;\;\; \text{if}\;\; 1 \leq r \leq p. 
\end{align}

\begin{defn} \label{def:F}
The following basic matrix kernels over $\hb$ will constitute a final kernel.

(0) Let $d > 0$. Define
\begin{align*}
    &F[{\scriptstyle 0}](r,u;s,v) = \ind{s < r^*} \oint\limits_{\Gamma_{-d}} d \zeta\, \frac{e^{\zeta(v-u)}}{\G(\zeta \vt \D_{s,r^*}(t,x,\xi))}.
\end{align*}
Recall $\Gamma_{d}$ is a vertical contour oriented upwards that intersects the real axis at $d$.

(1) Let $d_1 > 0$ and $D > 0$. Define
\begin{align*}
& F[{\scriptstyle p \vt p}]\, (r,u;s,v) = \ind{r=p} \, \oint\limits_{\Gamma_{-d_1}} d \zeta_1 \oint\limits_{\Gamma_{D}} d z_p\,
\frac{\G \big(z_p \vt \D_p(t,x,\xi) \big) \, e^{\zeta_1 v - z_p u}}{\G \big( \zeta_1 \vt \D_{s^{*},p}(t,x,\xi)\big) \, (z_p-\zeta_1)}.
\end{align*}

(2) Let $0 < d_1 < d_2$. For $1 \leq k \leq p$, define
\begin{align*}
& F[{\scriptstyle k,k \vt \emptyset}]\, (r,u; s,v)	=  \ind{ s < k < r^{*}} \, 
\oint\limits_{\Gamma_{-d_1}} d\zeta_1 \oint\limits_{\Gamma_{-d_2}} d\zeta_2 \\
& \frac{(\zeta_1 - \zeta_2)^{-1} \, e^{\zeta_2 v - \zeta_1 u}}
{\G \big(\zeta_1 \vt \D_{k,r^{*}}(t,x,\xi)\big) \, \G \big (\zeta_2 \vt \D_{s,k}(t,x,\xi) \big)}\,.
\end{align*}

The next two kernels are given in terms of integer parameters $0 \leq k_1 < k_2 \leq p$ and a vector parameter $\veps = (\eps_1, \ldots, \eps_{p-1}) \in \{1,2\}^{p-1}$.
Given $k_1$, $k_2$ and $\veps$, consider any set of distinct positive real numbers $D_k$ for integers $k \in (k_1, k_2]$ that
satisfy the following pairwise ordering for every $k \in (k_1,k_2)$:
\begin{equation} \label{Dorder}
D_k < D_{k+1}\;\; \text{if}\;\; \eps_k = 1\;\; \text{while}\;\; D_k > D_{k+1}\;\; \text{if}\;\; \eps_k = 2.
\end{equation}
It is easy to see that it is always possible to order distinct real numbers such that they satisfy
these constraints imposed by $\veps$. An explicit choice would be
$$D_1 = 2^p \quad \text{and} \quad D_{k+1} = D_k + (-1)^{\eps_k+1} \, 2^k.$$
Denote the contour
$$\vec{\Gamma}_{D^{\veps}} = \Gamma_{D_{k_1+1}} \times \cdots \times \Gamma_{D_{k_2}}.$$
(3)  Let $d_1 > 0$. Define
\begin{align*}
& F^{\veps}[{\scriptstyle k_1 \vt (k_1, k_2]}]\, (r,u; s,v) = \ind{k_1 < r^{*},\, s=k_2 < p,\, k_1<k_2}\,
\oint\limits_{\Gamma_{-d_1}} d\zeta_1 \, \oint\limits_{\vec{\Gamma}_{D^{\veps}}} dz_{k_1+1} \cdots dz_{k_2} \\
& \frac{\prod_{k_1 < k \leq k_2} \G \big(z_k \vt \D_k (t,x,\xi) \big)\, \prod_{k_1 < k < k_2} (z_k - z_{k+1})^{-1}\, e^{z_{k_2} v - \zeta_1 u}}
{\G \big(\zeta_1 \vt \D_{k_1,r^{*}} (t,x,\xi) \big) (z_{k_1+1}- \zeta_1)}.
\end{align*}

(4) Let $d_1, d_2 > 0$. Define
\begin{align*}
& F^{\veps}[{\scriptstyle k_1,k_2 \vt (k_1, k_2]}]\, (r,u;s,v) = \ind{k_1 < r^{*}, \, s^{*} < k_2,\, k_1<k_2}\,
\oint\limits_{\Gamma_{-d_1}} d\zeta_1 \oint\limits_{\Gamma_{-d_2}} d\zeta_2 \, \oint\limits_{\vec{\Gamma}_{D^{\veps}}} dz_{k_1+1} \cdots dz_{k_2} \\
& \frac{\prod_{k_1 < k \leq k_2} \G \big(z_k \vt \D_k (t,x,\xi)\big )\, \prod_{k_1 < k < k_2} (z_k - z_{k+1})^{-1}\, e^{\zeta_2 v - \zeta_1 u}}
{\G \big(\zeta_1 \vt \D_{k_1,r^{*}} (t,x,\xi) \big) \, \G \big(\zeta_2 \vt \D_{s^{*},k_2}(t,x,\xi) \big)\, (z_{k_1+1}-\zeta_1)\, (z_{k_2}-\zeta_2)}.
\end{align*}
\end{defn}

Using these kernels, consider the following kernels obtained as weighted sums.
Let $\theta_1, \ldots, \theta_{p-1}$ be
non-zero complex numbers and $\mathbold{\theta} = (\theta_1,\ldots, \theta_{p-1})$. Recall $\theta(r \vt \veps)$
and $\Theta(r \vt k)$ from \eqref{thetaeps} and \eqref{theta}, respectively. Define the following kernels over $\hb$.

\begin{align*}
F^{(0)}(r,u; s,v) & = (1+ \Theta(r \vt s)) \cdot F[{\scriptstyle 0}]\, (r,u;s,v).\\
F^{(1)}(r,u; s,v) & = \sum_{1 \leq k \leq p} \Theta(r \vt k)\cdot F[{\scriptstyle k,k \vt \emptyset}]\, (r,u;s,v).
\end{align*}
The final kernel involves variables $k_1, k_2 \in \{0,\ldots, p\}$ and $\veps \in \{1,2\}^{p-1}$. They satisfy
\begin{equation} \label{sumcond}
k_1 < k_2; \quad \text{given}\;\; k_1, k_2,\; \veps = (\overbrace{2, \ldots, 2}^{\substack{\eps_i = 2 \;\text{if} \\ i < \max \{k_1,1\}}},
\overbrace{\eps_{\max \{k_1,1\}}, \ldots, \eps_{\min\{k_2,p-1\}}}^{\text{arbitrary 1 or 2}},
\overbrace{1,\ldots, 1}^{\substack{\eps_i = 1 \;\text{if} \\ i > \min\{k_2,p-1\}}}).
\end{equation}
Recall the notation $(-1)^{\eps_{[k_1,k_2]}}$ following \eqref{theta}. Define
\begin{align*}
& F^{(2)}(r,u;s,v)  = \sum_{\substack{k_1, k_2,\, \veps \\ \text{satisfies}\; \eqref{sumcond}}}
(-1)^{\eps_{[k_1,k_2]} + \ind{k_2 =p}} \cdot \theta(r \vt \veps) \, \times \\
& \quad  \Big [  F^{\veps}[{\scriptstyle k_1 \vt (k_1,k_2]}] + F^{\veps}[{\scriptstyle k_1,k_2 \vt (k_1,k_2]}] +
\ind{k_1=p-1,\,k_2=p} F[{\scriptstyle p \vt p}] \Big ] (r,u;s,v). 
\end{align*}

Finally, define the kernel
\begin{equation} \label{F}
F(\mathbold{\theta}) = F^{(0)} + F^{(1)} + F^{(2)}.
\end{equation}

\paragraph{\textbf{Conjugation factor}}
We need to surround the kernel $F(\mathbold{\theta})$ with a conjugation to ensure that its entries decay rapidly; the series expansion of its Fredholm determinant will then be convergent absolutely. Define the following multiplication operator $\Upsilon$ on $\hb$.
\begin{equation} \label{eqn:conj}
    (\Upsilon f)(r,u) = e^{ \kappa_r |u| + \mu \cdot \mathrm{sgn}(u) |u|^{3/2}} \cdot f(r,u), \quad 1 \leq r \leq p; \; u \in \R.
\end{equation}
The numbers $\kappa_r$ need to be increasing: $\kappa_r < \kappa_s$ when $r < s$. For instance, $\kappa_r = r$ would do.
The number $\mu > 0$ is a sufficient small constant whose value will be stipulated during the proof (essentially $0 < \mu < (2/3)t_p^{-1/2}$ would do to balance the decay rate of the Airy function).

\begin{thm} \label{thm:droplet}
Suppose $p \geq 2$. Let $0 < t_1 < t_2 < \cdots < t_p$, $x_1, \cdots, x_k \in \R$ and $\xi_1, \cdots, \xi_p \in \R$.
Then,
\begin{equation*}
    \pr{\mathcal{L}(0,0;x_i,t_i) \leq \xi_i; 1 \leq i \leq p} =
\oint\limits_{\gamma_r} d\theta_1 \cdots \oint\limits_{\gamma_r} d\theta_{p-1}\, \frac{\dt{I + \Upsilon^{-1} F(\mathbold{\theta})\Upsilon}_{\hb}}{\prod_k (\theta_k - 1)}
\end{equation*}
where $\gamma_r$ is a counter-clockwise circular contour around the origin of radius $r > 1$ and $F(\mathbold{\theta})$ is from \eqref{F}.
\end{thm}

\subsection{Outline of the article}
We will prove Theorem \ref{thm:droplet} first in Section \ref{sec:droplet}.
The proof follows \cite{JR1} with a new observation which is Proposition \ref{prop:L0goodness}.
In order to prove Theorem \ref{thm:stat}, we will at first establish in Section \ref{sec:limit} a limit theorem for the two-time law in a certain PNG model with boundary sources.
Theorem \ref{thm:stat} will then be derived in Section \ref{sec:stat} by taking an appropriate limit of the latter.

\section*{Acknowledgements}
I thank Kurt Johansson for discussions on the topic of this article.
I also had preliminary discussions with Yuchen Liao and Janosch Ortmann, and I thank them both.
This work began in earnest during the program "Random Matrices and Scaling Limits" at the Institut Mittag-Leffler, and I thank the institute for its hospitality.

\section{Proof of Theorem \ref{thm:droplet}} \label{sec:droplet}
Consider the geometric PNG model \eqref{eqn:G} with the choice of parameters:
$$ a_i = b_j = \sqrt{q}, \quad 0 < q < 1.$$
In Theorem 2 of \cite{JR1}, a formula is derived for the probability
\begin{equation} \label{ptime}
\pr{\mathbold{G}(m_1,n_1) < a_1, \mathbold{G}(m_2, n_2) < a_2, \ldots, \mathbold{G}(m_p,n_p) < a_p},
\end{equation} 
where $m_1 < m_2 < \cdots < m_p$ and $n_1 < n_2 < \cdots < n_p$. On taking a scaling limit of this probability, one gets a formula for the KPZ droplet.

For a large parameter $T$, consider $m_k$, $n_k$ and $a_k$ of the form (ignoring rounding)
\begin{align} \label{eqn:KPZscaling}
	n_k  &= t_kT-c_1x_k(t_kT)^{\frac{2}{3}}, \\ \nonumber
	m_k &= t_k T+c_1 x_k(t_kT)^{\frac{2}{3}}, \\ \nonumber
	a_k  &= c_2t_kT+c_3\xi_k(t_kT)^{\frac{1}{3}}.
\end{align}
The parameters above are
$0 < t_1 < t_2 < \cdots < t_p$, $x_1, x_2, \ldots, x_p \in \R$ and $\xi_1, \xi_2, \ldots, \xi_p \in \R$.
The $c_i$s are constants given according to
\begin{equation}\label{scalingconstants}
c_1=q^{-\frac{1}{6}}(1+\sqrt{q})^{\frac{2}{3}},\quad c_2=\frac{2\sqrt{q}}{1-\sqrt{q}},\quad c_3=\frac{q^{\frac{1}{6}}(1+\sqrt{q})^{\frac{1}{3}}}{1-\sqrt{q}},
\end{equation}
One is interested in the large $T$ limit of \eqref{ptime} with this scaling.
Introduce also the constants
\begin{equation} \label{eqn:scaling2}
c_0=q^{-\frac{1}{3}}(1+\sqrt{q})^{\frac{1}{3}}, \quad c_4 = \frac{q^{1/3}(1-\sqrt{q})}{(1+\sqrt{q})^{1/3}},
\end{equation}
which will appear during the proof.

\subsection{Multi-point formula for discrete PNG}
We quote Theorem 2 of \cite{JR1}, which provides a formula for $\eqref{ptime}$. First, we need to introduce notation.

\paragraph{\textbf{Delta notation}}
For integers $0 \leq k_1 < k_2 \leq p$, and $y$ being $m, n$ or $a$ from \eqref{eqn:KPZscaling}, define
\begin{equation} \label{delta}
\D_{k_1, k_2} y = y_{k_2} - y_{k_1} \quad \text{and} \;\; \D_k y = y_k - y_{k-1}.
\end{equation}
By convention, $y_0 = 0$ for $y = n,m,a$.
We will also use the shorthand
\begin{equation*}
\D_{k_1,k_2} (y^1, \ldots, y^{\ell}) = (\D_{k_1,k_2}y^1, \ldots, \D_{k_1,k_2}y^{\ell}) \quad \text{and}\;\;
\D_k (y^1, \ldots, y^{\ell}) = (\D_k y^1, \ldots, \D_k y^{\ell}).
\end{equation*}

\paragraph{\textbf{Block notation}}
The matrices that appear will have a $p \times p$ block structure
with the rows and columns partitioned according to
$$\{ 1, 2, \ldots, n_p \} = (0,n_1] \cup (n_1, n_2] \cup \cdots \cup (n_{p-1}, n_p].$$
Recall 
\begin{align}
r^{*} &= \min \{r,p-1\} \;\;\; \text{if}\;\; 1 \leq r \leq p. 
\end{align}
For an $n_p \times n_p$ matrix $M$, $1 \leq i,j \leq n_p$ and $1 \leq r,s \leq p$, write
\begin{equation} \label{eqn:blocknot}
	M(r,i ; s,j) = \mathbf{1}_{\big\{ i \in (n_{r-1}, \, n_r], \; j \in (n_{s-1}, \, n_s] \big\} } \cdot M(i,j).
\end{equation}
This is the $p \times p$ block structure of $M$ according to the partition of rows and columns above.

\paragraph{\textbf{Complex integrands}}
Define, for $n,m,a \in \Z$ and $w \in \C \setminus \{0,1-q,1\}$,
\begin{equation}\label{gnmx}
G^{*}(w \vt n,m,a)=\frac{w^n(1-w)^{a+m}}{\big(1-\frac{w}{1-q}\big)^m},
\end{equation}
as well as the function
\begin{equation} \label{Gnmx}
G(w \vt n,m,a) = \frac{G^{*}(w \vt n,m,a)}{G^{*} \big (1-\sqrt{q} \vt n,m,a \big)}.
\end{equation}
The number $ w_c = 1- \sqrt{q} $ is the critical point around which we will perform steepest descent analysis.

\begin{defn}
    Define the following $n_p \times n_p$ matrices according to their block structure.
    Denote by $\gamma_{\tau}$ a circular contour around 0 of radius $\tau < 1-\sqrt{q}$.
    Denote by $\gamma_R(1)$ a circular contour around 1 of radius $R$ with $q < R < \sqrt{q}$.

    (0) Define
    \begin{align*}
&L_{\scriptstyle{0}}(r,i; s,j) =
\frac{\ind{ s < r^{*}}}{1-\sqrt{q}} \oint\limits_{\gamma_{\tau}} d\zeta \, \frac{1}{G\big(\zeta \vt i-j+1, \, \D_{s, r^{*}}(m,a) \big)}.
\end{align*}

(1) Define
\begin{align*}
    L[{\scriptstyle p \vt p}](r,i;s,j) & = \frac{\ind{r=p}}{1-\sqrt{q}} \oint\limits_{\gamma_{\tau}} d\zeta_2 \oint_{\gamma_{R}(1)} dz_p\,
	 \frac{G \big(z_p \vt n_p-i, \D_p m, \D_p a\big)}{G\big (\zeta_2 \vt n_p-j+1, m_p-m_{s^*}, a_p-a_{s^*}\big)\, (z_p-\zeta_2)}.
\end{align*}

(2) Let $\tau_2 < \tau_1 < 1-\sqrt{q}$. Define, for $1 \leq k \leq p$,
\begin{align*}
    &L [{\scriptstyle k,k \vt \emptyset}]\, (r,i;s,j) = \frac{\ind{s < k < r^{*}}}{1-\sqrt{q}}\, \oint\limits_{\gamma_{\tau_1}} d\zeta_1\,  \oint\limits_{\gamma_{\tau_2}} d \zeta_2\,
    \frac{G \big(\zeta_1 \vt n_k - i, m_k- m_{r^*}, a_k-a_{r^*} \big)}{G \big(\zeta_2 \vt n_k - j+1, m_k-m_s, a_k-a_s\big) \,(\zeta_1 - \zeta_2)}.
\end{align*}

The following two family of matrices are defined for $0 \leq k_1 < k_2 \leq p$ and $\veps  \in \{1,2\}^{p-1}$. Let $\tau_2 < \tau_1 < 1-\sqrt{q}$.
Consider radii $R_{k_1+1}, \ldots, R_{k_2}$ such that $q < R_k < \sqrt{q}$ for every $k$, and they are ordered in the following way:
$$ R_k < R_{k+1} \;\; \text{if}\;\; \eps_k = 2 \quad \text{while} \quad R_k > R_{k+1} \;\;\text{if}\;\; \eps_k = 1.$$
Note this depends only on $\eps_{k_1+1}, \ldots, \eps_{k_2-1}$.
It is possible to arrange the radii according to $\veps$ as shown in \cite{JR1}.
Set $\vec{\gamma}_{R^{\veps}} = \gamma_{R_{k_1+1}}(1) \times \cdots \times \gamma_{R_{k_2}}(1)$.

(3) Define
\begin{align*}
&L^{\veps} [{ \scriptstyle k_1 \vt (k_1, k_2]}] (r,i;s,j) =  \frac{\ind{k_1 < r^{*}, \; s = k_2 < p,\, k_1<k_2}}{1-\sqrt{q}}\, \oint\limits_{\gamma_{\tau_1}} d\zeta_1 \oint\limits_{\vec{\gamma}_{R^{\veps}}}\, dz_{k_1+1} \cdots d z_{k_2} \\
& \frac{\prod_{k_1 < k < k_2} G \big(z_k \vt \D_k (n,m,a)\big) G \big( z_{k_2} \vt j-1-n_{k_2-1}, \, \D_{k_2}(m, a)\big) \,
\left (\frac{1-\zeta_1}{1-z_1}\right)^{\ind{k_1=0}}}
{G\big(\zeta_1 \vt  i-n_{k_1}, m_{r^*}-m_{k_1}, a_{r^*}-a_{k_1} \big) \prod_{k_1 < k < k_2}(z_k - z_{k+1}) \, (z_{k_1 +1} - \zeta_1)}.
\end{align*}

(4) Define 
\begin{align*}
&L^{\veps} [{\scriptstyle k_1, k_2 \vt (k_1, k_2]}] (r,i;s,j)  =  \frac{\ind{k_1 < r^{*}, \; s^{*} < k_2,\, k_1<k_2}}{1-\sqrt{q}}\,
\oint\limits_{\gamma_{\tau_1}} d \zeta_1 \oint\limits_{\gamma_{\tau_2}} d \zeta_2 \,
\oint\limits_{\vec{\gamma}_{R^{\veps}}} d z_{k_1+1} \, d z_{k_1+2} \cdots d z_{k_2} \\
& \frac{\prod_{k_1 < k \leq k_2} G \big ( z_k \vt \D_k (n,m,a) \big) \, \prod_{k_1 < k < k_2} (z_k - z_{k+1})^{-1}
\left (\frac{1-\zeta_1}{1-z_1}\right)^{\ind{k_1 =0}}(z_{k_1+1}-\zeta_1)^{-1} (z_{k_2}-\zeta_2)^{-1}}
{G\big(\zeta_1 \vt  i-n_{k_1}, m_{r^*}-m_{k_1}, a_{r^*}-a_{k_1} \big) \, G \big(\zeta_2 \vt n_{k_2}-j+1, m_{k_2}-m_{s^*}, a_{k_2}-a_{s^*} \big)}.
\end{align*}
\end{defn}

Let $\theta_1, \ldots, \theta_{p-1}$ be
non-zero complex numbers and $\mathbold{\theta} = (\theta_1,\ldots, \theta_{p-1})$. Recall $\theta(r \vt \veps)$
and $\Theta(r \vt k)$ from \eqref{thetaeps} and \eqref{theta}, respectively. Define the following $n_p \times n_p$ matrices.

\begin{align*}
L^{(0)}(r,u; s,v) & = (1+ \Theta(r \vt s)) \cdot L[{\scriptstyle 0}]\, (r,u;s,v).\\
L^{(1)}(r,u; s,v) & = \sum_{1 \leq k \leq p} \Theta(r \vt k)\cdot L[{\scriptstyle k,k \vt \emptyset}]\, (r,u;s,v).
\end{align*}
The final matrix involves $k_1, k_2 \in \{0,\ldots, p\}$ and $\veps \in \{1,2\}^{p-1}$. They must satisfy \eqref{sumcond}.
Recall the notation $(-1)^{\eps_{[k_1,k_2]}}$ following \eqref{theta}. Define
\begin{align*}
& L^{(2)}(r,u;s,v)  = \sum_{\substack{k_1, k_2,\, \veps \\ \text{satisfies}\; \eqref{sumcond}}}
(-1)^{\eps_{[k_1,k_2]} + k_1 + k_2^*} \cdot \theta(r \vt \veps) \, \times \\
& \quad  \Big [  L^{\veps}[{\scriptstyle k_1 \vt (k_1,k_2]}] + L^{\veps}[{\scriptstyle k_1,k_2 \vt (k_1,k_2]}] +
\ind{k_1=p-1,\,k_2=p} L[{\scriptstyle p \vt p}] \Big ] (r,u;s,v). 
\end{align*}

Finally, define the kernel
\begin{equation} \label{L}
L(\mathbold{\theta}) = L^{(0)} + L^{(1)} + L^{(2)}.
\end{equation}

The following theorem is proved in \cite{JR1}.
\begin{thm} \label{thm:dpng}
Suppose $p \geq 2$. For $m_1 < m_2 < \cdots < m_p$ and $n_1 < n_2 < \cdots < n_p$,
\begin{align*}
&\pr{\mathbold{G}(m_1,n_1) < a_1, \mathbold{G}(m_2,n_2) < a_2, \ldots, \mathbold{G}(m_p,n_p) < a_p} = \\
& \quad \oint\limits_{\gamma_r^{p-1}} d\theta_1 \cdots d\theta_{p-1}\,
\frac{\dt{I + L(\mathbold{\theta})}}{\prod_{k=1}^{p-1} (\theta_k -1)}.
\end{align*}
Here, $\gamma_r^{p-1} = \gamma_r \times \cdots \times \gamma_r$ ($p-1$ times) and $\gamma_r$
is a counter-clockwise, circular contour around the origin of radius $r > 1$.
\end{thm}

By taking the large $T$ limit of the formula from Theorem \ref{thm:dpng} under the scaling \eqref{eqn:KPZscaling}, we will establish Theorem \ref{thm:droplet}.

\subsection{Setting for asymptotic}
Consider the space $X = \{1, 2, \ldots, p\} \times \R$ and the measure $\lambda$ on it defined by
$ \int\limits_{X} d\lambda \, f = \sum_{k=1}^{p-1} \int_{-\infty}^{0} dx\, f(k,x) \; + \; \int_{0}^{\infty} dx\, f(p,x)$.
Define the Hilbert space
\begin{equation} \label{Hspace}
\hb = L^2(X,\lambda) \cong \underbrace{L^2(\R_{<0}, dx) \oplus \cdots \oplus L^2(\R_{<0},dx)}_{p-1} \, \oplus \,L^2(\R_{>0},dx)\,.
\end{equation}
Recall the partition $\{1\, \ldots, n_p\} = (0,n_1] \cup \cdots (n_{p-1},n_p]$.

An $n_p \times n_p$ matrix $M$ embeds as a kernel $\widetilde{M}$ on $\hb$ by
\begin{equation} \label{matrixembed}
\widetilde{M}(r,u ; \, s,v) = M \big(r, \, n_{\min \{r,p-1\}} + \lceil u \rceil; \; s, \, n_{\min \{s,p-1\}} + \lceil v \rceil \big).
\end{equation}
Here we have used the block notation \eqref{eqn:blocknot}.

By design,
$$  \dt{I+M}_{n_p \times n_p} = \dt{I + \widetilde{M}}_{\hb}$$
where the latter determinant is according to the Fredholm series expansion
$$ \dt{I + \widetilde{M}}_{\hb} = 1 + \sum_{k\geq 1} \frac{1}{k!} \int\limits_{X^k} d\lambda(r_1, u_1) \cdots d \lambda(r_k, u_k)\,
\dt{\widetilde{M}(r_i,u_i;\, r_j,u_j)}_{k \times k}.$$

In order to perform asymptotics we should rescale variables of $\widetilde{M}$ according to KPZ scaling \eqref{eqn:KPZscaling}.
In this regard, define $$\nu_T = c_0 T^{1/3}.$$ We change variables $(r,u) \mapsto (r, \nu_T \cdot u)$
in the Fredholm determinant of $\widetilde{M}$ above. So if we define a new kernel
\begin{equation} \label{Fredholmembed}
M_T(r,u ;\, s,v) = \nu_T \, \widetilde{M} \big ( r,\, \nu_T \cdot u;\; s,\, \nu_T \cdot v\big ),
\end{equation}
then
$$ \dt{I+ M_T}_{\hb} = \dt{I + M}_{n_p \times n_p}.$$

\begin{defn} \label{kerneldef}
Let $M_1, M_2, \ldots$, be a sequence of matrices understood
in terms of the $p \times p$ block structure above. Let $\widetilde{M}_N$ be the embedding of $M_N$
into $\hb$ as in \eqref{matrixembed}, and $M_T$ the rescaling according to \eqref{Fredholmembed}.
\begin{itemize}
	\item The kernels $M_T$ are \emph{good} if there are non-negative, bounded and integrable functions
	$g_1(x)$, $\ldots, g_p(x)$ on $\R$ such that following holds. For every $T$,
	$$ |M_T(r;u,\, s,v)| \leq g_r(u) g_s(v) \quad \text{for every}\;\; 1\leq r,s \leq p \;\; \text{and}\;\; u,v \in \R.$$
	\item The matrices are \emph{convergent} if there is a integral kernel $F$ on $\hb$ such that
	the following holds uniformly in $u,v$ restricted to compact subsets of $\R$.
	$$ \lim_{T \to \infty}\, M_T(r,u; \,s,v) = F(r,u;\, s,v) \quad \text{for every}\;\; 1\leq r,s \leq p.$$
\end{itemize}
\end{defn}
A straightforward consequence of the definitions, dominated convergence theorem and Hadamard's inequality is: if $M_1, M_2, \ldots$ are good and convergent with limit $F$ on $\hb$ then
$$\dt{I+ M_T}_{\hb} \to \dt{I+F}_{\hb} < \infty.$$
The kernel $F$ satisfies the same goodness bound as its approximants.

\subsection{Convergence of kernels}
The following proposition is proved in \cite{JR1} as Proposition 5.1.

\begin{prop} \label{prop:oldlimits}
The matrices $L^{(1)}$ and $L^{(2)}$ are convergent to $F^{(1)}$ and $F^{(2)}$, respectively, due to the following. Suppose $0 \leq k_1 < k_2 \leq p$.
\begin{enumerate}
\item The matrix $L^{\veps}[{\scriptstyle k_1, k_2 \vt (k_1, k_2]}]$ is convergent with limit
$(-1)^{k_2-k_1} \, F^{\veps}[{\scriptstyle k_1, k_2 \vt (k_1, k_2]}]$.
\item The matrix $L^{\veps}[{\scriptstyle k_1 \vt (k_1, k_2]}]$ is convergent with limit
$(-1)^{k_2-k_1} \, F^{\veps}[{\scriptstyle k_1 \vt (k_1, k_2]}]$.
\item The matrix $L[{\scriptstyle k,k \vt \emptyset}]$ is convergent with limit $F[{\scriptstyle k,k \vt \emptyset}]$.
\item The matrix $L[{\scriptstyle p \vt p}]$ is convergent with limit $- F[{\scriptstyle p \vt p}]$.
\end{enumerate}
\end{prop}
The exact same proof argument gives
\begin{prop} \label{prop:newlimit}
   The matrix  $L^{(0)}$ is convergent with limit $F^{(0)}$.
\end{prop}

\subsection{Decay estimates (goodness) for kernels}
We need to show that $\Upsilon L_T(\mathbold{\theta}) \Upsilon^{-1}$ is a good kernel.
We recall Lemma 5.3 of \cite{JR1}. Define the function
$$ \Psi(y) = e^{- \mu_1 (y)^{3/2}_{-} + \mu_2(y)_{+}}$$
where $(y)_{-} = \max \{-y,0\}$, $(y)_{+} = \max \{y,0\}$. The constants $\mu_1, \mu_2 > 0$ depend only on the parameters $q$ and $\max_i \{ |x_i|, |\xi_i|\}$ and are given in the proof of the lemma (see Lemma 5.7 of \cite{JoTwo2}). Suppose $n,m,a$ take the form
\begin{align*} 
 n &= K - c_1 x K^{2/3} + c_0 y K^{1/3}\\
 m &= K + c_1 x K^{2/3}\\
 h &= c_2 K + c_3 \xi K^{1/3}
\end{align*}
According to Lemma 5.3 of \cite{JR1}, there is a choice of circular contour $\gamma = \gamma(\sigma \vt K,y)$ around zero parametrized by $\sigma \in \R$ such that if $\zeta = \zeta(\sigma) \in \gamma$, $K$ is sufficiently large and $y$ is such that $n \geq 0$, then
\begin{equation} \label{eqn:globaldecaydroplet}
| G(\zeta(\sigma) \vt n,m,a)|^{-1} \leq C e^{-C \sigma^2 + \Psi(y)}.
\end{equation}
Here $C$ is an absolute constant.

\begin{prop} \label{prop:L0goodness}
    The kernels $\Upsilon^{-1} L_T[{\scriptstyle 0}] \Upsilon$ are good if $\mu$ is chosen to satisfy $0 < \mu < \mu_1 t_p^{-1/2}$ and the numbers $\kappa_i$ are increasing in $i \in \{1, \ldots, p\}$.
\end{prop}

\begin{proof}
    We may use \eqref{eqn:globaldecaydroplet} with $(n,m,a) = \D_{s, r^{*}}(n, m,a)$, $K = (\D_{s, r^*}t) T$ and $y = (u-v)/\D$ for $\D = (\D_{s, r^*}t)^{1/3}$.
    Then, by the definition of $\Upsilon$, we find that for an absolute constant $C'$,
    $$ |\Upsilon^{-1} L_T[{\scriptstyle 0}] \Upsilon(r,u;s,v)| \leq C' \ind{s < r^*}\, e^{-\kappa_r |u| - \mu \mathrm{sgn}(u) |u|^{3/2} + \kappa_s |v| + \mu \mathrm{sgn}(v) |v|^{3/2} + \Psi((u-v)/\D)}.$$
    We have that $\Psi((u-v)/\D) = e^{-\mu_1 (\D_{s, r^*}t)^{-1/2}(u-v)_{-}^{3/2} + \mu_2 (\D_{s, r^*}t)^{-1/3}(u-v)_{+}}$.
    Suppose $\mu \in (0, \mu_1 t_p^{-1/2})$. Then $\mu < \mu_1 (\D_{s, r^*}t)^{-1/2}$ because $\D_{s, r^*}t \leq t_p$. Since $(y)_{-} = (-y)_{+}$ and $(y)_{+} = (-y)_{-}$, we find that
    \begin{align*}
        &|\Upsilon^{-1} L_T[{\scriptstyle 0}] \Upsilon(r,u;s,v)| \leq \\
        & C' \ind{s < r^*}\, e^{-\kappa_r |u| - \mu \mathrm{sgn}(u) |u|^{3/2} + \kappa_s |v| + \mu \mathrm{sgn}(v) |v|^{3/2} -\mu_1 (\D_{s, r^*}t)^{-1/2}(v-u)_{+}^{3/2} + \mu_2 (\D_{s, r^*}t)^{-1/3}(v-u)_{-}}.
    \end{align*}
    Choose any $\kappa = \kappa_{r,s} \in (\kappa_s, \kappa_r)$, which is possible since $\kappa_s < \kappa_r$ due to $s < r^*$. Note that $v \leq 0$ since $s < r^* < p$.
    The bound above gives
    \begin{equation*}
        |\Upsilon^{-1} L_T[{\scriptstyle 0}] \Upsilon(r,u;s,v)| \leq g_s(v) g_r(u,v)
    \end{equation*}
    where $g_s(v) = e^{(\kappa_s - \kappa)|v|}$ and
    $$ g_r(u,v) = C' e^{-\kappa_r |u| - \mu \mathrm{sgn}(u) |u|^{3/2} + \kappa |v| - \mu |v|^{3/2} -\mu_1 (\D_{s, r^*}t)^{-1/2}(v-u)_{+}^{3/2} + \mu_2 (\D_{s, r^*}t)^{-1/3}(v-u)_{-}}$$
    where we have used the fact that $v \leq 0$. Clearly, $g_s(v)$ is bounded and integrable. We bound $g_r(u,v)$ by
    $$ g_r(u,v) \leq g_r(u) = \sup_{v \leq 0} g_r(u,v).$$
    The supremum is obtained at $v = u + O_{\kappa, \kappa_r,\mu_1,\mu}(1)$ or at $v = O_{\kappa, \kappa_r,\mu_1,\mu}(1)$ (depending on the sign of $u$).
    In either case, we find that there is a constant $C_{\kappa, \kappa_r,\mu_1,\mu}$ such that
    $$ g_r(u) \leq C_{\kappa, \kappa_r,\mu_1,\mu} e^{(\kappa - \kappa_r) |u| + (\mu - \mu_1 (\D_{s, r^*}t)^{-1/2})u_{-}^{3/2} - \mu u_{+}^{3/2} + \mu_2 (\D_{s, r^*}t)^{-1/3} u_{+}},$$
    which is bounded and integrable. The lemma follows.  
\end{proof}

\begin{prop} \label{prop:L12goodness}
    The kernels $\Upsilon^{-1} L^{(1)}_T \Upsilon$ and $\Upsilon^{-1} L^{(2)}_T \Upsilon$ are good if $\mu$ is chosen to satisfy $0 < \mu < \mu_1 t_p^{-1/2}$.
\end{prop}

\begin{proof}
    It is enough to show that $\Upsilon^{-1} L_T \Upsilon$ are good kernels for $L = L^{\veps}[{\scriptstyle k_1, k_2 \vt (k_1, k_2]}], L^{\veps}[{\scriptstyle k_1 \vt (k_1, k_2]}], L[{\scriptstyle k,k \vt \emptyset}]$ or $L[{\scriptstyle p \vt p}]$. For each of these 4 matrices, Proposition 5.1 on \cite{JR1} proves the following
    bound:
    $$ |L_T(r,u;s,v)| \leq e^{ \Psi(u/\Delta_1) + \Psi(-v/\Delta_2)}$$
    were $\Delta_1 = (\D_{a,b} t)^{-1/3}$ and $\Delta_2 = (\D_{c,d} t)^{-1/3}$ for certain indices $a < b$ and $c<d$ that depend on $r,s$ as well as $k_1,k_2$.
    As a result we find that
    $$ |\Upsilon^{-1} L_T \Upsilon(r,u;s,v)| \leq g_r(u) g_r(v)$$
    where
    $$g_r(u) = e^{-\mu \mathrm{sgn}(u) |u|^{3/2} - \mu_1 \Delta_1^{-1/2}u_-^{3/2} + \mu_2 \Delta_1^{-1/3}u_+ - \kappa_r |u|} =
    e^{(\mu - \mu_1 \Delta_1^{-1/2})u_-^{3/2} -\mu u_+^{3/2} + \mu_2 \Delta_1^{-1/3}u_+ - \kappa_r|u|},$$
    and
    $$g_s(v) = e^{\mu \mathrm{sgn}(v) |v|^{3/2} - \mu_1 \Delta_2^{-1/2}(-v)_-^{3/2} + \mu_2 \Delta_2^{-1/3}(-v)_+ + \kappa_s |v|} =
    e^{(\mu - \mu_1 \Delta_2^{-1/2})v_+^{3/2} -\mu v_-^{3/2} + \mu_2 \Delta_2^{-1/3}v_- + \kappa_s|v|}.$$
    Observe that $\mu_1 \D_i^{-1/2} \geq \mu_1 t_p^{-1/2}$, so $\mu - \mu_1 \D_i^{-1/2} < 0$. Also, $\D_i > 0$. As a result, $g_r$ and $g_s$ are both bounded and integrable.
\end{proof}

Propositions \ref{prop:L0goodness} and \ref{prop:L12goodness} combine to complete the proof of Theorem \ref{thm:droplet}.

\begin{rem} \label{rem:otherconj}
    One may use other conjugation factors just as effectively. Another valid choice would be to have $\Upsilon f(r,u) = (1+ |u|)^{\kappa_r} e^{\mu u}$,
    where $\mu$ is sufficiently large in terms of the parameters $(x_i,t_i)$ and the constants $\kappa_r$ satisfy $\kappa_r - \kappa_s > 2$ for every $r>s$.
    The proofs of Propositions \ref{prop:L0goodness} and \ref{prop:L12goodness} go through in the same manner.
\end{rem}

\section{A limit theorem for PNG with boundary sources} \label{sec:limit}
Let $T$ be a large scaling parameter. For $0 < q < 1$, recall the constants \eqref{scalingconstants} and \eqref{eqn:scaling2}. Consider the geometric PNG model \eqref{eqn:G} with the following choice of parameters:
\begin{equation} \label{eqn:albe}
a_1 = 1-\frac{\alpha}{c_3} T^{-1/3}, \; b_1 = 1 - \frac{\beta}{c_3} T^{-1/3}, \; a_i = b_j = \sqrt{q} \quad \text{for}\; i,j > 1.
\end{equation}
The parameters $\alpha$ and $\beta$ are positive and fixed. One views this choice of parameters as the homogeneous PNG model with sources at the boundary corresponding to $\alpha$ and $\beta$. The weights along the first row and column of the quadrant are heavier than the bulk, and the optimal path $\pi$ whose weight attains the value $\Gb(T,T)$ will spend an order $T^{2/3}$ many steps on the boundary before venturing into the bulk. The choices of $a_1$ and $b_1$ are made in a  manner so that the weight $\omega(1,1)$ survives in the upcoming scaling limit to become an Exponential random variable of rate $\alpha + \beta$.
This and similar models have been studied in \cite{BP,IS, PS}.

We present the two-time law of this model under KPZ scaling via its cumulative distribution function.
In terms of the KPZ fixed point, this presents the law of the pair $(X_1 + \omega, X_2 + \omega)$, where $\omega \sim \mathrm{Exp}(\alpha + \beta)$ is independent of $(X_1,X_2)$ and $X_i = \mathcal{L}(h_0; (x_i,t_i))$ with
$\mathcal{L}(h_0;(x,t))$ being the KPZ fixed point with the random initial condition
$$h_0(y) = B(y) -\alpha(y)_{-} - \beta (y)_{+}.$$
Here $B$ is a two-sided Brownian motion with diffusivity 2 ($\E{B(y)^2} = 2|y|$). 
In the sequel we will remove the weight $\omega$ to get a formula for the law of $(X_1,X_2)$, and then take the limit $\alpha,\beta \to 0$ to get the two-time law of the KPZ fixed point with Brownian initial condition.

The two-time distribution function of the geometric PNG model has been derived in \cite{JR2}. We quote that theorem.
\begin{thm} \label{thm:0}
The two-time distribution function with general parameters $a_i,b_j$ and $n<N, m<M$ is given by
$$\pr{\Gb(m,n) \leq h, \Gb(M,N) \leq H} = \frac{1}{2 \pi \mathbold{i}} \oint \limits_{|\theta| = s > 1}
\frac{\dt{I + \theta^{\ind{i > n}} F_1 + \theta^{-\ind{i \leq n}} F_2}}{\theta - 1}.$$
These matrices are sums, $F_1 = J_1 - J_2 + J_3$ and $F_2 = J_2 - J_1 - J_4$, with the $J$s given by the following formulas.

For $z \in \C$, $h \in \Z$ and subsets $S \subset [N]$ and $T \subset [M]$, define
\begin{equation} \label{eqn:GST}
	G(z \vt S, T, h) = z^{h} \prod_{k \in S} (z - 1/b_k) \prod_{k \in T} (1 - a_k/z)^{-1}.
\end{equation}

Recall that $\gamma$ is a circular contour, oriented counter-clockwise.
Recall also that we use the notation $\oint_{\gamma}$ to mean $(2 \pi \mathbold{i})^{-1} \oint_{\gamma}$.

\begin{align*}
J_1(i,j) &= \ind{j \leq n} \, \oint_{\gamma_b} d \zeta \, \oint_{\gamma_a} dz \
\frac{G (z \vt [j-1], [m], h-1)}{G (\zeta \vt [i], [m], h-1 )\, (z-\zeta)}\\
J_2(i,j) & = \ind{i > n} \, \oint_{\gamma_b} d\omega \, \oint_{\gamma_a} dw \,
\frac{G(w \vt [N] \setminus [i], [M] \setminus [m], H-h)}{G(\omega \vt [N]\setminus [j-1], [M] \setminus [m], H-h) \, (w-\omega)}\\
J_3(i,j) & = \oint_{\gamma_b} d\zeta \, \oint_{\gamma_b} d \omega \, \oint_{\gamma_a} dz\,  \oint_{\gamma'_a} dw  \\
& \frac{G(z \vt [n],[m], h-1) G(w \vt [N]\setminus [n], [M] \setminus [m], H-h)}
{G(\zeta \vt [i],[m], h) \, G(\omega \vt [N]\setminus [j-1], [M] \setminus [m], H-h)\,(z-\zeta)(w-\omega)(z-w)}
\end{align*}
The contour $\gamma_b$ encloses only the poles at every $1/b_k$. The contours $\gamma_a$ and $\gamma'_a$ enclose only
the poles at every $a_k$. In $J_3$, $\gamma_a$ contains $\gamma'_a$ (so $|z| > |w|$).

The matrix $J_4$ looks the same as $J_3$ except the $z$ and $w$ contours are reversed
so that $\gamma'_a$ contains $\gamma_a$ (so $|w| > |z|$).
\end{thm}

Take the geometric PNG model with choice of parameters \eqref{eqn:albe}.
Write $n,N, m, M, h$ and $H$ according to the following scaling. Consider
temporal parameters $0 < t_1 < t_2$, spatial parameters $x_1, x_2 \in \R$ and height parameters $\xi_1, \xi_2 \in \R$.
For a choice of these, set (ignoring rounding)
\begin{align} \label{kpzscaling}
& n = t_1 T - c_1 x_1 (t_1T)^{2/3} & N = t_2 T - c_1 x_2(t_2T)^{2/3} \\ \nonumber
& m = t_1 T + c_1 x_1 (t_1T)^{2/3} & M = t_2 T + c_1 x_2(t_2T)^{2/3} \\ \nonumber
& h = c_2 (t_1 T) + c_3 \xi_1(t_1T)^{1/3} & H = c_2 (t_2 T) + c_3 \xi_2(t_2 T)^{1/3} .
\end{align}

Introduce the notation $\D n = N-n$, $\D m = M-m$ and $\D h = H-h$. Recall \eqref{eqn:deltas}.
It holds that $\D n = \D t T - c_1 \D x(\D tT)^{2/3}$ and likewise for $\D m$ and $\D h$.

\subsection{Statement of the limit theorem}
Recall the function $\G$ from \eqref{Gtxx}.
Let $d_1, d_2, D_1$ and $D_2$ be positive real numbers such that
$$ d_1, d_2 <  \beta; \quad D_1, D_2 < \alpha.$$
Denote by $\Gamma_d$ the vertical contour crossing the real axis at $d$ and oriented upwards.
Let $\mu$ be a sufficiently large scalar that will be used in a conjugation factor.

Define the following four kernels over $L^2(\R)$.
\begin{align} \label{eqn:Js} \nonumber
& \J_{1}(u,v) = e^{\mu(v-u)} \, \ind{v \leq 0}\,  \oint \limits_{\Gamma_{-d_1}} d \zeta \oint \limits_{\Gamma_{D_1}} dz \\ \nonumber
& \qquad \frac{\G(z \vt t_1, x_1, \xi_1) e^{zv-\zeta u}}{\G(\zeta \vt t_1, x_1, \xi_1)} \frac{(\zeta-\alpha)(z+\beta)}{(z-\zeta)(z-\alpha)(\zeta+\beta)}\\ \nonumber
& \J_{2}(u,v) = e^{\mu(v-u)} \, \ind{u > 0}\,  \oint \limits_{\Gamma_{-d_1}} d \zeta \oint \limits_{\Gamma_{D_1}} dz\, 
\frac{\G(z \vt \D t, \D x, \D \xi) e^{\zeta v - zu}}{\G(\zeta \vt \D t, \D x, \D \xi) (z-\zeta)} \\ \nonumber
& \J_{3}(u,v) =  \; e^{\mu(v-u)} \, \oint \limits_{\Gamma_{-d_1}} d \zeta \oint \limits_{\Gamma_{-d_2}} d \omega 
\oint \limits_{\Gamma_{D_1}} dz  \oint \limits_{\Gamma_{D_2}} dw \\ \nonumber
& \qquad \frac{\G(z \vt t_1, x_1, \xi_1) \, \G(w \vt \D t, \D x, \D \xi)}{\G(\zeta \vt t_1, x_1, \xi_1) \, \G(\omega \vt \D t, \D x, \D \xi)} \cdot
\frac{e^{\omega v - \zeta u} (\zeta-\alpha)(z+\beta)}{(z-\alpha)(\zeta+\beta)(z-\zeta)(w-\omega)(z-w)}. \\ \nonumber
& \J_{4}(u,v) =  \; e^{\mu(v-u)} \, \oint \limits_{\Gamma_{-d_1}} d \zeta \oint \limits_{\Gamma_{-d_2}} d \omega 
\oint \limits_{\Gamma_{D_1}} dz  \oint \limits_{\Gamma_{D_2}} dw \\ \nonumber
& \qquad \frac{\G(z \vt t_1, x_1, \xi_1) \, \G(w \vt \D t, \D x, \D \xi)}{\G(\zeta \vt t_1, x_1, \xi_1) \, \G(\omega \vt \D t, \D x, \D \xi)} \cdot
\frac{e^{\omega v - \zeta u} (\zeta-\alpha)(z+\beta)}{(z-\alpha)(\zeta+\beta) (z-\zeta)(w-\omega)(z-w)}.
\end{align}
In $\J_3$, $D_1 < D_2$, that is, the $z$-contour is to the left of the $w$-contour. In $J_4$, $D_1 > D_2$, so that the ordering of the contours is reversed.
The kernels are of trace class if $\mu$ is sufficiently large in terms of $x_1,x_2,t_1$ and $t_2$.

\begin{thm} \label{thm:1}
	Consider the two-time distribution $\pr{G(m,n) < h, G(M,N) < H}$ for the geometric PNG model
	with choice of parameters \eqref{eqn:albe}
	Assume that $n,m,h,N,M,H$ are given by \eqref{kpzscaling}. Then in the limit as $T$ tends to infinity,
	the two time distribution function converges to
	$$ \frac{1}{2 \pi \mathbold{i}} \oint \limits_{|\theta| = r} \frac{\dt{I + F(\theta)}_{L^2(\R)}}{\theta - 1}$$
	where $r > 1$ and
	$$F(\theta)(u,v) = \theta^{\ind{u > 0}} F_{1}(u,v) + \theta^{- \ind{u \leq 0}} F_{2}(u,v).$$
	The kernels $F_{1}$ and $F_{2}$ are given by
	\begin{equation*}
	F_{1} = \J_2 - \J_1  + \J_3 \quad \text{and} \quad F_{2} = \J_1 - \J_2 - \J_4\,.
	\end{equation*}
\end{thm}

\subsection{Proof of the theorem}
The proof follows from Theorem \ref{thm:0} by a saddle point analysis of the determinantal kernels. 
We must show that the matrices $F_1$ and $F_2$ from Theorem \ref{thm:0}, under KPZ-scaling
\eqref{eqn:Fkpz} with the parameters scaled according to \eqref{eqn:albe}, converge to the corresponding matrices $F_1$ and $F_2$
from Theorem \ref{thm:1}.

\subsubsection{Embedding}
Embed an $N \times N$ matrix $M$ (where $n$ and $N$ are the parameters from the two-time distribution)
as a kernel over $L^2(\R)$ by the formula
$$M \mapsto F(u,v) = M(n + \lceil u \rceil, n + \lceil v \rceil)$$
where $u, v \in \R$. Set $F(u,v)$ to be zero when $n + \lceil u \rceil$ or $n + \lceil v \rceil$ lie outside the set $[N]$.
Then, it follows readily that
$$\dt{I + M}_{N \times N} = \dt{I + F}_{L^2(\R)},$$
where the latter determinant should be taken as the Fredholm series expansion of $F$.
The KPZ re-scaled kernel is defined to be
\begin{equation} \label{eqn:Fkpz}
F_T(u,v) = \nu_T \cdot F(\nu_T u, \, \nu_T v) \quad \text{with}\; \nu_T = c_0 T^{1/3}.
\end{equation}
Note that $\dt{I+F}_{L^2(\R)}$ equals $\dt{I+F_T}_{L^2(\R)}$.
The matrices $F_1$ and $F_2$ from Theorem \ref{thm:0} will be considered under the scaling \eqref{eqn:Fkpz}.

It suffices to show that the matrices $J_i$ are good and converge to their respective limits $\J_i$ (according to Definition \eqref{kerneldef}.
We will carry out the procedures above for the matrix $J_3$ and show that it converges to $\J_3$.
The steps are similar for the other $J$ matrices; $J_1$ converges to $-\J_1$, $J_2$
converges to $-\J_2$ and $J_4$ converges to $\J_4$.
We will omit these for brevity.

\subsubsection{Descent contours}
Consider circular contours $\gamma_0$, around 0, and $\gamma_1$, around 1, as contours for the integration variables $\zeta, z, \omega, w$.
Firstly, define
\begin{equation} \label{eqn:wc}
w_c = 1- \sqrt{q},
\end{equation}
which is the critical point around which asymptotics will be performed. Now, for a (large) parameter $K$, define
\begin{align} \label{eqn:contours}
\gamma_0 &= \gamma_0(\sigma, d) = w_c(1- \frac{d}{K^{1/3}}) e^{\mathbold{i} \sigma K^{-1/3}} \quad |\sigma| \leq \pi K^{1/3}, \\ \nonumber
\gamma_1 &= \gamma_1(\sigma, d) = 1 - \sqrt{q}(1- \frac{d}{K^{1/3}}) e^{\mathbold{i} \sigma K^{-1/3}} \quad |\sigma| \leq \pi K^{1/3}.
\end{align}
The parameter $d$ should satisfy $0 < d < K^{1/3}$. Observe that if $\sigma$ remains bounded independently of $K$ then
one has the expansions
$$ \gamma_0(\sigma, d) = w_c + w_c\frac{(\mathbold{i} \sigma - d)}{K^{1/3}} + O(K^{-2/3}), \quad
\gamma_1(\sigma, d) = w_c + \sqrt{q} \frac{(-\mathbold{i}\sigma + d)}{K^{1/3}} + O(K^{-2/3}).$$
So, locally around $\sigma = 0$, the contours are vertical lines.
\smallskip

\subsubsection{Re-expressing the matrix $J_3$}
Define the function
\begin{equation} \label{eqn:gstar}
    G^*(z \mid n,m,h) = z^n (1-z)^{m+h} \big (1 - \frac{z}{1-q}\big )^{-m}
\end{equation}
Define also
\begin{equation} \label{eqn:GG}
    G(z \mid n,m,h) = \frac{G^*(z \mid n,m,h)}{G^*(w_c\mid n,m,h)}
\end{equation}
which is $G^*$ normalized around the critical point.

In the contour integral defining the kernel $J_3$ in Theorem \ref{thm:0},
make the change of variables $z \mapsto q^{-1/2}(1-z)$ and same for the other variables $\zeta, \omega, w$.
After some bookkeeping, the matrix $J_3$ looks as follows.

\begin{align*}
    J_3(i,j) &= (-q)^{(j-i)/2} \oint_{\gamma_1} d\zeta \, \oint_{\gamma_2} d \omega \, \oint_{\gamma_3} dz\,  \oint_{\gamma_4} dw  \\
& \frac{G^*(z \vt n-1,m-1, h-1) G^*(w \vt \D n, \D m, \D h)}
{G^*(\zeta \vt i-1, m-1, h-1) \, G^*(\omega \vt N-j+1, \D m, \D h)} \\
\times & \frac{(z+\sqrt{q}b^{-1}-1)(1-\zeta-\sqrt{q}a)(1-z)}{(\zeta+\sqrt{q}b^{-1}-1)(1-z-\sqrt{q}a) (1-\zeta)(z-\zeta)(w-\omega)(z-w)}
\end{align*}
The contours are circular, oriented counter-clockwise and should not intersect.
The contour $\gamma_1$ encloses the poles at $0$ and $1-\sqrt{q}b^{-1}$ in $\zeta$.
The contour $\gamma_2$ encloses the pole at $0$ in $\omega$. The contour $\gamma_3$ encloses the poles at $1-\sqrt{a}$ and $1-q$ in $z$.
The contour $\gamma_4$ encloses the pole at $1-q$ in $w$. Finally, we require that $|1-z| > |1-w|$ throughout the contours.

The factor $(-q)^{(j-i)/2}$ is a conjugation and appears in every $J$ matrix. So it can be removed from the determinant.
We also may replace $G^*$ by its normalization $G$ in the integrand. Doing so introduces another conjugation factor $w_c^{j-i}$, which may be removed from the determinant. Thus, we may express $J_3$ as follows.
\begin{align} \label{eqn:J3}
    J_3(i,j) &= \oint_{\gamma_1} d\zeta \, \oint_{\gamma_2} d \omega \, \oint_{\gamma_3} dz\,  \oint_{\gamma_4} dw  \\ \nonumber
& \frac{G(z \vt n-1,m-1, h-1) G(w \vt \D n, \D m, \D h)}
{G(\zeta \vt i-1, m-1, h-1) \, G(\omega \vt N-j+1, \D m, \D h)} \\ \nonumber
\times & \frac{(z+\sqrt{q}b^{-1}-1)(1-\zeta-\sqrt{q}a)(1-z)}{(\zeta+\sqrt{q}b^{-1}-1)(1-z-\sqrt{q}a) (1-\zeta)(z-\zeta)(w-\omega)(z-w)}
\end{align}

\subsubsection{Convergence of $J_3$}
Under KPZ scaling, the indices $i$ and $j$ are written as $i = n + \lceil \nu_T u \rceil$ and $j = n + \lceil \nu_T v \rceil$
for $u, v \in \R$ and $\nu_T = c_0 T^{1/3}$. The KPZ rescaled kernel for $J_3$ is
$$ J_T(u,v) = \nu_T \cdot J_3(n + \lceil \nu_T u \rceil, n + \lceil \nu_T v \rceil ).$$

We now choose contours for each of the integration variables.
\begin{align*}
& \zeta = \zeta(\sigma_1) \in \gamma_0 \left ( \frac{c_4}{w_c} \sigma_1, \frac{c_4}{w_c} d_1 \right)
& z = z(\sigma_2) \in \gamma_1\left ( \frac{c_4}{\sqrt{q}} \sigma_1, \frac{c_4}{\sqrt{q}} D_1 \right) 
& \quad K = t_1T, \\
& \omega = \omega(\sigma_3) \in \gamma_0 \left ( \frac{c_4}{w_c} \sigma_3, \frac{c_4}{w_c} d_2 \right)
& w = w(\sigma_4) \in \gamma_1\left ( \frac{c_4}{\sqrt{q}} \sigma_4, \frac{c_4}{\sqrt{q}} D_2 \right)
& \quad K =\D t T.
\end{align*}
We need to have $D_1/t_1^{1/3} < D_2/(\D t)^{1/3}$ in order to satisfy the constraint $|1-z| > |1-w|$.
We also need that $d_1 < \beta t_1^{1/3}$ and $D_1 < \alpha t_1^{1/3}$ to ensure all necessary poles are included inside the contours.

By Lemma 5.3 of \cite{JR1}, which gives decay estimates for the integrand along the chosen contours, we have the estimate
$$ \nu_T \cdot \left ( \text{integrand of}\; J_3(n + \lceil \nu_T u \rceil, n + \lceil \nu_T v \rceil) \right )
\leq C_1 e^{- C_2 (\sigma_1^2+\sigma_2^2+\sigma_3^2+\sigma_4^2)},$$
so long as $u$ and $v$ remain bounded and where $C_1$ and $C_2$ are constants that depend on $u,v$ and the parameters $t_i,x_i,\xi_i$. This allows us to use the dominated convergence theorem to
find the limiting integral for $J_T(u,v)$ by considering its point-wise limit with $u$, $v$ and the $\sigma_k$s held fixed.

With the $\sigma_k$ held fixed, Taylor expansion gives
\begin{align*}
& \zeta(\sigma_1) = w_c + \frac{c_4}{(t_1T)^{1/3}} (\mathbold{i} \sigma_1 - d_1) + C_{q,L}T^{-\frac{2}{3}}
& z(\sigma_2) = w_c + \frac{c_4}{(t_1T)^{1/3}} (\mathbold{i} \sigma_2 + D_1) + C_{q,L}T^{-\frac{2}{3}} \\
& \omega(\sigma_3) = w_c + \frac{c_4}{(\D t T)^{1/3}} (\mathbold{i} \sigma_3 - d_2) + C_{q,L}T^{-\frac{2}{3}}
& w(\sigma_4) = w_c + \frac{c_4}{(\D t T)^{1/3}} (\mathbold{i} \sigma_4 + D_2) + C_{q,L}T^{-\frac{2}{3}}.
\end{align*}
Here $C_{q,L}$ is a constant that depends on $q$ and a large $L$ for which $|t_i|, |x_i|, |\xi_i| \leq L$.

Write
$$\zeta' = (\mathbold{i}\sigma_1 -d_1)/ t_1^{1/3}, z' = (\mathbold{i}\sigma_2 + D_1)/t_1^{1/3},
\omega' = (\mathbold{i}\sigma_3-d_2)/(\D t)^{1/3}, w' = (\mathbold{i} \sigma_4 + D_2)/(\D t)^{1/3}.$$
In these new variables, at $T$ tends to infinity, the contours become
vertical lines. The contours of $\zeta'$ and $\omega'$ become, respectively, the lines $\Re(\zeta') = -d_1/t_1^{1/3}$
and $\Re(\omega') = -d_2/(\D t)^{1/3}$, oriented upwards. The $z$-contour becomes $\Re(z') = D_1/t_1^{1/3}$, oriented downwards.
The $w$-contour becomes downwardly oriented $\Re(w') = D_2/(\D t)^{1/3}$. If these contours are then oriented upwards,
we obtain a factor of $(-1)^2 = 1$.

Set $d'_1 = d_1/t^{1/3}$, $d'_2 = d_2/(\D t)^{1/3}$ and $D'_1 = D_1/t_1^{1/3}$, $D'_2 = D_2/(\D t)^{1/3}$.
The constraints on the limiting contours become $d'_1, d'_2 > 0$, $d_1' < \beta$, $0 < D'_1 < D'_2$ and $D_1' < \alpha$.

Next, we consider the behaviour of the integrand along these contours.
By Lemma 5.2 of \cite{JR1} (which gives the local behaviour of $G$ around $w_c$ under KPZ scaling), if
$$n = K - c_1 x K^{2/3} + c_0 u K^{1/3}, \, m = K + c_1 xK^{2/3}, \, h = c_2K + c_3 \xi K^{1/3}$$
and $w = w_c + (c_4/K^{1/3}) w'$, then uniformly for $w'$ in any compact set, 
$$\lim_{K \to \infty} G(w \vt n,m,h) = \G(w' \vt 1, x, \xi - u) = \exp \{ w^3/3 + x w^2 - (\xi-u)w\}.$$
Consequently, as $T \to \infty$, one has that (recall \eqref{Gtxx}):
\begin{align*}
& G(z \vt n, m, h) \to \G(t_1^{1/3}z' \vt 1,x_1, \xi_1) = \G(z' \vt t_1, x_1, \xi_1), \\
& G(w \vt \D n, \D m, \D h) \to \G((\D t)^{1/3}w' \vt 1, \D x, \D \xi) = \G(w' \vt \D t, \D x, \D \xi),\\
& G(\zeta \vt i-1, m-1, h-1) \to \G(t_1^{1/3}\zeta' \vt 1,x_1, \xi_1 - t_1^{-1/3}u) = \G(\zeta' \vt t_1, x_1, \xi_1)e^{\zeta' u}, \\
& G(\omega' \vt N+1-j, \D m, \D h) \to \G((\D t)^{1/3}\omega' \vt 1, \D x, \D \xi + (\D t)^{-1/3}v) = \G(\omega' \vt \D t, \D x, \D \xi)e^{-\omega' v}.
\end{align*}

A calculation shows that
$$\frac{\nu_T}{w_c}\, \frac{d \zeta \, d \omega \, d z \, d w}{(z-\zeta)(w-\omega)(z-w)} =
\frac{d \zeta' \, d \omega' \, d z' \, d w'}{(z'-\zeta')(w'-\omega')(z'-w')} + C_{q,L}T^{-1/3}.$$
Also, $(1-z)/(1-\zeta)$ tends to 1.
Finally, consider the term
$$ \frac{(z+\sqrt{q}b^{-1} - 1)(1-\zeta - \sqrt{q}a)}{(\zeta+\sqrt{q}b^{-1} - 1)(1-z - \sqrt{q}a)},$$
which equals
$$ \frac{z' + (\sqrt{q}/c_3 c_4)\beta}{\zeta' + (\sqrt{q}/c_3 c_4)\beta} + O(T^{-1/3}) = \frac{z' + \beta}{\zeta' + \beta} + O(T^{-1/3}).$$
Similarly,
$$\frac{1-\zeta - \sqrt{q}a}{1 - z - \sqrt{q}a} = \frac{\zeta' - \alpha}{z' -\alpha} + O(T^{-1/3}).$$

Therefore, for $u,v$ restricted to any compact set, $J_T(u,v)$ converges to
\begin{align*}
& \oint \limits_{\Gamma_{-d'_1}} d \zeta' \oint \limits_{\Gamma_{-d'_2}} d \omega'
\oint \limits_{\Gamma_{D'_1}} dz'  \oint \limits_{\Gamma_{D'_2}} dw' \\
& \frac{\G(z' \vt t_1, x_1, \xi_1) \, \G(w' \vt \D t, \D x, \D \xi)}{\G(\zeta' \vt t_1, x_1, \xi_1) \, \G(\omega' \vt \D t, \D x, \D \xi)} \cdot
\frac{e^{\omega' v - \zeta' u} (\zeta-\alpha)(z + \beta)}{(z-\alpha)(\zeta+\beta)(z'-\zeta')(w'-\omega')(z'-w')}.
\end{align*}
The constraint on the contours is that $d'_1, d'_2 > 0$, $d_1' < \beta$, $0 < D'_1 < D'_2$ and $D_1' < \alpha$.
This limit is $\J_3$ but without the conjugation factor $e^{\mu (v-u)}$.

\subsubsection{Decay estimate (goodness) for $J_3$}
In order to show $J_3$ form a good sequence of kernels, it is necessary to include the
conjugation factor $e^{\mu (v-u)}$ for a sufficiently large constant $\mu$ in front of
the KPZ re-scaled kernel $J_T(u,v)$.

By Lemma 5.3 of \cite{JR1}, one has the following estimates where $C_1$ and $C_2$ are constants
that depend only on $q$ and $L$ (recall all parameters $t_i, x_i$ and $\xi_i$  are bounded in absolute value by $L$).
\begin{align*}
& |G(\zeta(\sigma_1) \vt i-1, m-1, h)|^{-1} \leq C_1 e^{-C_2 \sigma_1^2 + \Psi(u)}\\
& |G(\omega(\sigma_3) \vt N+1-j, \D m, \D h)|^{-1} \leq C_1 e^{-C_2 \sigma_3^2 + \Psi(-v)} \\
& |G(z(\sigma_2) \vt n-1, m-1, h)| \leq C_1 e^{-C_2 \sigma_2^2} \\
& |G(w(\sigma_4) \vt \D n, \D m, \D h)| \leq C_1 e^{-C_2 \sigma_4^2}.
\end{align*}
Here $\Psi(x) = - \mu_1 (x)_{-}^{3/2} + \mu_2(x)_{+}$ for some positive constants $\mu_1$ and $\mu_2$.

There is also a constant $C_3$ that depends on $q$ and $\alpha, \beta$ such that
$$ \left | \frac{(z + \sqrt{q}b^{-1} -1)(1-\zeta-\sqrt{q}a)(1-z)}{(\zeta + \sqrt{q}b^{-1} -1)(1-z-\sqrt{q}a)(1-\zeta)(z-\zeta)(w-\omega)(z-w)} \right | \leq C_3$$
It follows from these estimates that for the kernel $J_T$,
$$e^{\mu(v-u)} |J_T(u,v)| \leq C_{q,L, \alpha, \beta} \, e^{-\mu u + \Psi(u)} \cdot e^{\mu v + \Psi(-v)}.$$
Finally, observe that for $\mu > \max \{\mu_1, \mu_2 \}$, the function $e^{-\mu x + \Psi(x)}$
is bounded and integrable. This shows the required decay estimate and completes the proof.

\section{Proof of Theorem \ref{thm:stat}} \label{sec:stat}

We shall prove this theorem in several steps. Firstly, by Theorem \ref{thm:1},
\begin{equation} \label{eqn:cdf0}
    \pr{\mathcal{L}(h_0; x_i,t_i) + \omega \leq \xi_i, i=1,2} = \frac{1}{2 \pi \mathbold{i}} \oint \limits_{|\theta| = r} d\theta\, \frac{\dt{I + F(\theta)}_{L^2(\R)}}{\theta - 1}.
\end{equation}
Here $\omega$ is an Exponential random variable with rate $\alpha + \beta$ that is independent of both $\mathcal{L}(h_0; x_i,t_i)$, $h_0$ is the initial condition $h_0(y) = B(y) - \alpha (y)_{-} -\beta(y)_{+}$ and $F$ is the kernel from Theorem \ref{thm:1}. We have to remove the weight $\omega$ and then take the limit $\alpha, \beta \to 0$ to get the desired formula for the CDF of $\mathcal{L}(B; x_i,t_i)$.

\subsection{Removing the weight -- the shift argument}
The way to remove the weight $\omega$ is explained in \cite{FS} and \cite{BFP} in what is called the shift argument. Their arguments are carried out in the exponential last passage percolation setting, but there is one additional technicality that we need to handle in our setting.
Let
\begin{equation} \label{eqn:PP} 
P^+(\xi_1,\xi_2) = \pr{\mathcal{L}(h_0; x_i,t_i) + \omega \leq \xi_i, i=1,2}, \quad P(\xi_1,\xi_2) = \pr{\mathcal{L}(h_0; x_i,t_i) \leq \xi_i, i=1,2}
\end{equation}
The shift argument presents $P$ in terms of $P^+$. In Exponential last passage percolation, $P^+(\xi_1,\xi_2) = 0$ if either $\xi_1$ or $\xi_2$ are negative. This is simply because all last passage values are positive. This fact is used in the shift argument. In our case, $P^+$ may be positive for negative $\xi_i$, so we need an auxiliary decay estimate on $P^+(\xi_1,\xi_2)$ as $\xi_i \to -\infty$.

\begin{lem} \label{lem:Pdecay}
    There are constants $C$ and $\kappa_1, \kappa_2 > 0$ such that
    $$ P^+(\xi_1,\xi_2) \leq C e^{-\kappa_1 (\xi_1)_{-}^2 -\kappa_2 (\xi_2)_{-}^2}.$$
    The same estimate hold for $P(\xi_1,\xi_2)$.
\end{lem}

\begin{proof} 
    Since the weight $\omega \geq 0$, $P^+(\xi_1,\xi_2) \leq P(\xi_1,\xi_2)$. With $h_0(y) = B(y) - \alpha (y)_{-} -\beta(y)_{+}$ we have, by the variational formula for the KPZ fixed point,
    $$P(\xi_1,\xi_2) = \pr{h_0(y_1) + \mathcal{L}(0,y_1;x_1,t_1) \leq \xi_1,  h_0(y_2) + \mathcal{L}(0,y_2;x_2,t_2) \leq \xi_2\; \text{for all}\; y_1,y_2 \in \R}$$
    Therefore, by Cauchy-Schwartz,
    $$P(\xi_1,\xi_2) \leq \pr{h_0(y) + \mathcal{L}(0,y;x_1,t_1) \leq \xi_1\; \text{for all}\; y}^{1/2} \cdot \pr{h_0(y) + \mathcal{L}(0,y;x_2,t_2) \leq \xi_2\; \text{for all}\; y}^{1/2} .$$
    Define $P_i(\xi) = \pr{h_0(y) + \mathcal{L}(0,y;x_i,t_i) \leq \xi\; \text{for all}\; y}$ for $i=1,2$. It is enough to show that $P_i(\xi) \leq C e^{-\kappa_i \xi^2}$
    for $\xi \leq 0$. Set $L_1 = \mathcal{L}(0,-1; x_1,t_1)$ and $L_2 = \mathcal{L}(0,1; x_2,t_2)$.
    We find that
    $$P_1(\xi) \leq \pr{B(-1) \leq \xi_1 + \beta - L_1}, \quad P_2(\xi) = \pr{B(1) \leq \xi_2 + \alpha - L_2}.$$
    Since $B(-1)$ and $B(1)$ are normal random variables independent of the $L_i$, upon conditioning on $(L_1,L_2)$, we see that
    \begin{align*}
        P_1(\xi) &\leq \E{ \pr{B(-1) \leq \xi + \beta - L_1} \mid L_1},\\
        P_2(\xi) &\leq \E{ \pr{B(1) \leq \xi + \alpha - L_2} \mid L_2}.
    \end{align*}
    Now, \begin{align*}
        \pr{B(-1) \leq \xi + \beta - L_1 \mid L_1} &= \pr{B(-1) \leq \xi + \beta - L_1, L_1 \geq \xi/2 \mid L_1} \\
        & \qquad + \pr{B(-1) \leq \xi + \beta - L_1, L_1 \leq \xi/2 \mid L_1} \\
        &\leq \pr{B(-1) \leq \xi/2 +\beta} + \pr{L_1 \leq \xi/2 \mid L_1}
    \end{align*}
    Therefore,
    \begin{align*}
       & P_1(\xi) \leq \pr{B(-1) \leq \xi/2 + \beta} + \pr{L_1 \leq \xi/2}, \\
       &P_2(\xi) \leq \pr{B(1) \leq \xi/2 + \alpha} + \pr{L_2 \leq \xi/2}.
    \end{align*}

    The random variable $B(-1)$ is a normal, so $\pr{B(-1) \leq \xi_1/2 +\beta} \leq C e^{-\kappa (\xi_1)_{-}^2}$ for some constant $C$ and $\kappa > 0$.
    The random variable $L_1$ is a scaled GUE Tracy-Widom, so $\pr{L_1 \leq \xi/2} \leq C' e^{-\kappa' (\xi)_{-}^3}$. We infer that there is a $\kappa_1 > 0$ and constant $C_1 < \infty$ such that $P_1(\xi) \leq C e^{-\kappa_1 (\xi)_{-}^2}$.
    By the same reasoning, $P_2(\xi) \leq C e^{-\kappa_2 (\xi)_{-}^2}$.
    Therefore, $$P^+(\xi_1,\xi_2) \leq P(\xi_1,\xi_2) \leq C e^{-\kappa_1 (\xi_1)_{-}^2 -\kappa_2 (\xi_2)_{-}^2}.$$
\end{proof}

\begin{lem} \label{lem:shift}
    The probabilities $P^+(\xi_1,\xi_2)$ and $P(\xi_1,\xi_2)$ are related by the identity
    $$ P(\xi_1,\xi_2) = \left ( 1 + \frac{1}{\alpha+\beta} (\partial_{\xi_1} + \partial_{\xi_2}) \right) P^+(\xi_1,\xi_2)$$
\end{lem}
\begin{proof}
    The proof is the same as Proposition 2.1 of \cite{BFP}, but using Lemma \ref{lem:Pdecay}. As $\omega$ is independent of $\mathcal{L}(B; x_i,t_i)$
    and has law $\mathrm{Exp}(\alpha + \beta)$, by conditioning on $\omega$ we have that
    $$ P^+(\xi_1,\xi_2) = \int_0^{\infty} dy\, P(\xi_1-y, \xi_2-y) r e^{-ry}, \quad r = \alpha + \beta.$$
    Consider the Laplace transform of $P^+$: for $s_1, s_2 > 0$
    \begin{align*} 
    I(s_1,s_2) &= \int_{-\infty}^{\infty} d\xi_1 \int_{-\infty}^{\infty} d\xi_2\, P^+(\xi_1,\xi_2) e^{-s_1\xi_1 - s_2\xi_2} \\
    &= r \int_{-\infty}^{\infty} d\xi_1 \int_{-\infty}^{\infty} d\xi_2\, \int_0^{\infty} dy\, P(\xi_1-y, \xi_2-y) e^{-ry - s_1\xi_1 -s_2\xi_2} \\
    &= r \int_{-\infty}^{\infty} dz_1 \int_{-\infty}^{\infty} dz_2\, \int_0^{\infty} dy\, P(z_1,z_2) e^{-y(r+s_1+s_2) -s_1z_1 -s_2z_2}\\
    &= \frac{r}{r + s_1 + s_2} \int_{-\infty}^{\infty} dz_1 \int_{-\infty}^{\infty} dz_2\,P(z_1,z_2) e^{-s_1z_1 -s_2z_2}
    \end{align*}
    The decay estimate from Lemma \ref{lem:Pdecay} ensures all integrals above are absolutely convergent.

    We deduce that
    $$ \int_{-\infty}^{\infty} dz_1 \int_{-\infty}^{\infty} dz_2\,P(z_1,z_2) e^{-s_1z_1 -s_2z_2} = 
    \int_{-\infty}^{\infty} d\xi_1 \int_{-\infty}^{\infty} d\xi_2\, (1 + \frac{s_1 + s_2}{r}) P^+(\xi_1,\xi_2)e^{-s_1 \xi_1 - s_2 \xi_2}.$$
    Integration by parts shows
    $$ \int_{-\infty}^{\infty} d\xi_1 s_1 P^+(\xi_1,\xi_2)e^{-s_1 \xi_1} = \int_{-\infty}^{\infty} d\xi_1 \partial_{\xi_1} P^+(\xi_1,\xi_2)e^{-s_1\xi_1}.$$
    So we infer that
    \begin{align*}
        & \int_{-\infty}^{\infty} dz_1 \int_{-\infty}^{\infty} dz_2\,P(z_1,z_2) e^{-s_1z_1 -s_2z_2} = \\
        & \int_{-\infty}^{\infty} d\xi_1 \int_{-\infty}^{\infty} d\xi_2\, (1 + \frac{1}{r} (\partial_{\xi_1} + \partial_{\xi_2})) P^+(\xi_1,\xi_2) e^{-s_1\xi_1 -s_2\xi_2}.
    \end{align*}
    By the uniqueness of the Laplace transform, $P(z_1,z_2) = (1 + \frac{1}{r} (\partial_{z_1} + \partial_{z_2})) P^+(z_1,z_2)$.
\end{proof}

\subsection{Auxiliary lemmas}

\begin{lem} \label{lem:Iab}
    For $a, b \in \R$, define
    $$I(a,b) = \int_{-\infty}^0 d\lambda\, e^{a \lambda} \Ai(\lambda + b).$$
    If $|a| \leq A$, then there is a constant $C_A$ such that
    $$ |I(a,b)| \leq C_A e^{- \frac{2}{3} (b)^{3/2}_{+}}, \quad (b)_{+} = \max \{0,b\}.$$
\end{lem}

\begin{proof}
    For $z \in \C$, for any $\delta > 0$,
   $$ \Ai(z) = \frac{1}{2 \pi \mathbold{i}} \oint \limits_{\Re(w)=\delta} dw\, e^{\frac{w^3}{3} - zw}.$$
    
    For $d \in \C$, let $W_d$ denote the wedge-shaped contour
    $$ W_d = \{ d + \mathbold{i}s e^{-\mathbold{i}\theta}; s\in (-\infty, 0]\} \cup \{ d + \mathbold{i}s e^{\mathbold{i}\theta}; s\in [0,\infty)\}. $$
    The contour is oriented counter-clockwise.
    Here $\theta \in (\pi/3, \pi/2)$. If $w = w(s) \in W_d$ then $\Re((w-d)^3) = |s|^3 \sin(3 \theta)$, and $\sin(3\theta) < 0$ due to the choice of $\theta$.
    Also, $\Re(w) = d - |s|\sin(\theta)$ with $\sin(\theta) > 0$.

    We may deform the contour $\Re(z) = \delta$ in the definition of $\Ai(z)$ to $W_{d}$ for any $d \in \R$. Let us choose a $d < a$. We have that
   $$I(a,b) = \int_{-\infty}^{0}d\lambda \frac{1}{2 \pi \mathbold{i}} \oint \limits_{W_d} dw\, e^{\frac{w^3}{3} - bw + \lambda(a-w)}.$$
    Since $\Re(a-w) = a-d + \sin(\theta) |s| > 0$, we can interchange the two integrals to obtain
    \begin{equation} \label{eqn:Iab}
    I(a,b) = \frac{1}{2 \pi \mathbold{i}} \oint \limits_{W_d} dw\, \frac{e^{\frac{w^3}{3}-bw}}{a-w}.
    \end{equation}

    If we also choose $d$ to be negative, then $\Re(-bw) = b|d| + b|s|\sin(\theta) \to -\infty$ as $b \to -\infty$. The dominated convergence theorem
    then implies that $\lim_{b \to -\infty} I(a,b) = 0$. Note that $I(a,b)$ is also continuous in $b$. Therefore, there is a constant $C'_A$ such that
    $$ |I(a,b)| \leq C'_A \quad \text{for all}\; b \leq A^2+1.$$

    Suppose $b \geq A^2 + 1$. Let $d' = \sqrt{b} + \mathbold{i}$ and consider $W_{d'}$. We can choose the angle $\theta = \theta_A \in (\pi/3,\pi/2)$
    sufficiently close to $\pi/2$ such that $W_{d'}$ contains $a$ in its exterior. Indeed, this will be the case so long as $W_{d'}$ intersects
    the real axis at a point $C < -A$. We have that $C = - \cos(\theta)^{-1}$, so the condition is satisfied when $\cos(\theta) < A^{-1}$,
    which will be the case if $\theta \approx \pi/2$. With such a choice to angle, we can translate the contour $W_d$ to $W_{d'}$ without
    crossing the the pole at $w = a$.

    Along the contour $W_{d'}$, $|w-a| \geq \sqrt{b} - a \geq \sqrt{A^2+1} - A >0$. We also have that $w = \sqrt{b} + \mathbold{i}t$,
    where $t = 1 + s \cos(\theta) + \mathbold{i}|s| \sin(\theta)$.
    A computation shows that $(\sqrt{b} + \mathbold{i}t)^3/3 - b(\sqrt{b} + \mathbold{i}t) = -\frac{2}{3} b^{3/2} - \sqrt{b}t^2 - \mathbold{i}\frac{t^3}{3}$.
    The real part of this equals $-\frac{2}{3} b^{3/2} - |s|^3 \frac{\sin(\theta)^3}{3} + O(s^2)$, from which we deduce that
    $$ |I(a,b)| \leq C_{\theta} e^{-\frac{2}{3}b^{3/2}} = C_A e^{- \frac{2}{3}b^{3/2}}.$$
\end{proof}

\begin{lem} \label{lem:Ibeta}
    Let $\beta > 0$ and $u \in \R$. Define
    $$ I_{\beta}(u) = e^{- \mu u} \frac{1}{2 \pi \mathbold{i}} \oint \limits_{\Re(\zeta) = -d} d \zeta\, \frac{e^{-\zeta u}}{\G(\zeta \vt t_1,x_1,\xi_1)(\zeta + \beta)}$$
    for any $0 < d < \beta$. Then, if $\mu > x_1 t^{-1/3}$, $\sup_{0 < \beta < 1} \int_{-\infty}^{\infty}du\,  |I_{\beta}(u)|^2 < \infty$.
\end{lem}

\begin{proof}
    After changing variables $\zeta \to -\zeta$ and reorienting the contour upwards, we have that
    $$ I_{\beta}(u) = e^{- \mu u} \frac{1}{2 \pi \mathbold{i}} \oint \limits_{\Re(\zeta) = d} d \zeta\, \frac{\G(\zeta \vt t_1,-x_1,\xi_1) e^{\zeta u}}{\beta-\zeta}.$$
    Since $\Re(\beta - \zeta) = \beta - d > 0$, we may write $(\beta-\zeta)^{-1} = \int_{-\infty}^0 d \lambda\, e^{\lambda(\beta - \zeta)}$. Therefore,
    $$ I_{\beta}(u) =\int_{-\infty}^0 d\lambda e^{\lambda \beta - \mu u}\,
    \frac{1}{2 \pi \mathbold{i}} \oint \limits_{\Re(\zeta) = d} d \zeta\, e^{\zeta (u-\lambda)}\G(\zeta \vt t_1,-x_1,\xi_1).$$
    From the contour integral representation of the Airy function $\Ai(\cdot)$, we find that
    $$\frac{1}{2 \pi \mathbold{i}} \oint \limits_{\Re(\zeta) = d} d \zeta\, e^{\zeta (u-\lambda)}\G(\zeta \vt t_1,-x_1,\xi_1) = 
    t_1^{-1/3} e^{-\frac{2}{3}x_1^3 - x_1(\xi_1 +t_1^{-1/3}(\lambda-u))} \Ai(\xi_1 + x_1^2 + t_1^{-1/3}(\lambda - u)).$$
    Consequently,
    $$I_{\beta}(u) = t_1^{-1/3} e^{-\frac{2}{3}x_1^3 -x_1 \xi_1 + (x_1t_1^{-1/3}-\mu)u} \int_{-\infty}^0 d \lambda\, e^{\lambda(\beta -x_1 t_1^{-1/3})}
    \Ai(\xi_1 + x_1^2 + t_1^{-1/3}(\lambda-u)).$$

    The integral above equals $I(a,b)$ from Lemma \ref{lem:Iab} with $a = t^{1/3}\beta - x_1$ and $b = \xi_1 + x_1^2 - t^{-1/3}u$. Since $0 < \beta < 1$,
    Lemma \ref{lem:Iab} implies there is a constant $C_{t_1,x_1,\xi_1}$ such that
    \begin{equation} \label{eqn:Ibeta} \sup_{0 < \beta < 1} |I_{\beta}(u)| \leq C_{t_1,x_1,\xi_1} e^{(x_1t_1^{-1/3}-\mu)u} e^{t^{-1/2}u^{3/2}_-}.\end{equation}
    Here $u_- = \max \{0,-u\}$. The right side above is square integrable over $u \in \R$ if $\mu > x_1 t_1^{-1/3}$, and so the lemma follows.
\end{proof}

\begin{lem} \label{lem:Ibetaprime}
    For $\beta > 0$ consider the quantity $I_{\beta}(u)$ from Lemma \ref{lem:Ibeta}. We have that
    $$ I_{\beta}(u) = e^{u(\beta-\mu)} \G(-\beta \vt t_1,x_1,\xi_1)^{-1} + I'_{\beta}(u)$$
    where, for $\mu > x_1 t_1^{-1/3}$, $\sup_{0 < \beta < 1} \int_{-\infty}^{\infty} |I'_{\beta}(u)|^2 < \infty$.
\end{lem}

\begin{proof}
    In the contour integral defining $I_{\beta}$, move the $\zeta$-contour to the left of $-\beta$ and collect a residue at $\zeta = -\beta$. Then we have
    $$I_{\beta}(u) = e^{u(\beta-\mu)} \G(-\beta \vt t_1,x_1,\xi_1)^{-1} + e^{-\mu u} \frac{1}{2 \pi \mathbold{i}} \oint \limits_{\Re(\zeta) = -d} d \zeta\, \frac{e^{-\zeta u}}{\G(\zeta \vt t_1,x_1,\xi_1)(\zeta + \beta)}$$
    with $d > \beta$. Set $I'_{\beta}$ to be the above integral. Changing variables $\zeta \to -\zeta$ and reorienting the contour upward gives
    $$I'_{\beta}(u) = e^{-\mu u} \frac{1}{2 \pi \mathbold{i}} \oint \limits_{\Re(\zeta) = d} d \zeta\, \frac{\G(\zeta \vt t_1,-x_1,\xi_1) e^{\zeta u}}{\beta-\zeta}$$
    but now with $\Re(\beta-\zeta) = \beta - d < 0$. Writing $(\beta-\zeta)^{-1} = - \int_{0}^{\infty} d\lambda \, e^{\lambda(\beta-\zeta)}$,
    and using the contour integral representation of the Airy function, we find that
    $$ I'_{\beta}(u) = - t_1^{-1/3} e^{- \frac{2}{3}x_1^3 + x_1\xi_1 + u(x_1 t_1^{-1/3}-\mu)} \int_0^{\infty} d\lambda\, e^{\lambda(\beta- x_1t_1^{-1/3})} \Ai(\xi-1 + x_1^2 + t_1^{-1/3}(\lambda-u)).$$

    Assume $\beta \in (0,1)$. The $\lambda$ integral is dominated by the decay of the Airy function for large $\lambda$. A saddle point analysis shows that the leading contribution comes for $\lambda \approx O(1)$. By choosing $\mu$ to be large, we can ensure that $x_1t_1^{-1/3}-\mu < 0$. Thus,
    $$ \sup_{\beta \in (0,1)} |I'_{\beta}(u)| \leq C_{t_1,x_1,\xi_1} e^{(x_1t_1^{-1/3}-\mu)u} |\Ai(-t_1^{-1/3} u)|, \quad u \in \R.$$
    The right side above is square integrable and so the lemma follows.
\end{proof}

\begin{lem} \label{lem:Iz}
For $z \in \C$ with $\Re(z) \geq 0$ and $v \in \R$, define
$$ I_z(v) = e^{\mu v} \frac{1}{2 \pi \mathbold{i}} \oint \limits_{\Re(\omega) = -d} d \omega\, \frac{e^{\omega v}}{\G(\omega \vt \D t,\D x,\D \xi) (z-\omega)}$$
with $d > 0$. For $\mu > \D x (\D t)^{-1/3}$, there is a constant $C$ that depends only of $\D t,\D x$ and $\D \xi$ such that
$$ |I_z(v)| \leq C e^{(\mu - \D x (\D t)^{-1/3})v} |\Ai((\D t)^{-1/3} v)|, \quad v \in \R.$$ 
\end{lem}

\begin{proof}
    Changing variables $\omega \to -\omega$ and reorienting the contour we have
    $$I_z(v) = e^{\mu v} \frac{1}{2 \pi \mathbold{i}} \oint \limits_{\Re(\omega) = d} d \omega\, \frac{\G(\omega \vt \D t,-\D x,\D \xi) e^{-\omega v}}{ z+\omega}.$$
    Since $\Re(z+\omega) > 0$, we write $(z+\omega)^{-1} = \int_{-\infty}^0 d \lambda e^{\lambda(z+\omega)}$. Then we have
    \begin{align*}
        I_z(v) &= e^{\mu v} \int_{-\infty}^0 d \lambda e^{\lambda z} \frac{1}{2 \pi \mathbold{i}} \oint \limits_{\Re(\omega) = d} d \omega\,
        \G(\omega \vt \D t, - \D x, \D \xi) e^{\omega(\lambda-v)} \\
        &= (\D t)^{-1/3} e^{- \frac{2}{3}(\D x)^{3} - \D x \D \xi  +(\mu - \D x (\D t)^{-1/3})v} \times \\
        & \int_{-\infty}^{0} d\lambda\, e^{\lambda(z + \D x (\D t)^{1/3})}
        \Ai(\D \xi + (\D x)^2 + (\D t)^{-1/3}(v-\lambda)).
    \end{align*}
    There is fast decay in $\lambda$ in the integrand above: the Airy function decays like $\Ai(-(\D t)^{-1/3} \lambda)$ for large, negative values of $\lambda$.
    Moreover, since $\Re(z) \geq 0$, $|e^{\lambda (z + \D x (\D t)^{1/3})}| \leq e^{\lambda \D x (\D t)^{1/3}}$. The main contribution to the integral comes from $\lambda \approx O(1)$. Upon choosing $\mu$ large such that $\mu - \D x (\D t)^{-1/3} >0$, we see that there is a constant $C = C_{\D t, \D x, \D \xi}$ such that
    \begin{equation} \label{eqn:Iz} |I_z(v)| \leq C e^{(\mu - \D x (\D t)^{-1/3})v} |\Ai((\D t)^{-1/3} v)|, \quad v \in \R.\end{equation}
\end{proof}

\subsection{Behaviour of $\J_i$ in the limit}
We investigate the behaviour of the kernels $\J_i$ from Theorem \ref{thm:1} as $\alpha, \beta \to 0$.
Note that $A_2$ from Theorem \ref{thm:stat} equals $J_2$ from Theorem \ref{thm:1}.

    For $d \in \R$, let $W_d$ denote the wedge-shaped contour
    $$ W_d = \{ d + \mathbold{i}s e^{-\mathbold{i}\theta}; s\in (-\infty, 0]\} \cup \{ d + \mathbold{i}s e^{\mathbold{i}\theta}; s\in [0,\infty)\}. $$
    The contour is oriented counter-clockwise, that is, from $-\infty \mathbold{i} e^{-\mathbold{i}\theta}$ to $\infty \mathbold{i} e^{\mathbold{i}\theta}$.
    Also, $\theta \in (\pi/3, \pi/2)$. If $w = w(s) \in W_d$ then $\Re((w-d)^3) = |s|^3 \sin(3 \theta)$, and $\sin(3\theta) < 0$ due to the choice of $\theta$.
    
\begin{prop} \label{prop:J1}
    Suppose $0 < \alpha, \beta < 1$. The kernel $\J_1$ admits the decomposition
    $$ \J_1 = A_1 + (\alpha + \beta) (f_1 \otimes g_1) + E,$$
    where $A_1$ is from Definition \ref{def:A} and
    \begin{align*}
        f_1(u) &= e^{-\mu u} \oint \limits_{\Gamma_{-d}} d\zeta\, \frac{e^{-\zeta u}}{\G(\zeta \vt t_1,x_1,\xi_1) (\zeta+\beta)},\\
        g_1(v) &= e^{(\mu+\alpha)v} \ind{v \leq 0} \G(\alpha \vt t_1,x_1,\xi_1).
    \end{align*}
    In the above, $d > 0$. If $\mu > \max \{0, x_1 t_1^{-1/3}\}$, then there is a constant $C$ such that $\sup_{0<\alpha,\beta < 1} || f_1 \otimes g_1||_{tr} \leq C$ and $||E||_{tr} \to 0$
    as $\alpha, \beta \to 0$.
\end{prop}

\begin{proof}
    In the definition of $\J_1$, we can deform the $\zeta$-contour from the vertical line $\Gamma_{-d_1}$ to $W_{\delta}$ for any $0 < \delta < D_1$.
    Next, push the $z$-contour ($\Gamma_{D_1}$) to the right of $\alpha$ (since $D_1 < \alpha)$, and obtain a residue at $z = \alpha$. This gives
    $ \J_1 = A_{\alpha,\beta} + B$, where
    $$ A_{\alpha,\beta}(u,v) = e^{\mu(v-u)} \ind{v \leq 0}\, \oint \limits_{W_{\delta}} d\zeta \oint \limits_{\Gamma_{D_1}} dz\,
        \frac{\G(z \vt t_1,x_1,\xi_1) e^{zv-\zeta u}(\zeta-\alpha)(z+\beta)}{\G(\zeta \vt t_1, x_1, \xi_1) (z-\zeta)(z-\alpha)(\zeta+\beta)},$$
    with $-\beta < 0 < \delta < \alpha < D_1$, and
    $$ B(u,v) = (\alpha +\beta) f_1(u) g_1(v).$$

    We may now take the limit $\alpha, \beta \to 0$ of $A_{\alpha,\beta}$ in the trace norm to find that $A_{\alpha,\beta} \to A_1$ where
    \begin{align*}
        A_1(u,v) &= e^{\mu(v-u)} \ind{v \leq 0}\, \oint \limits_{W_{\delta}} d\zeta \oint \limits_{\Gamma_{D}} dz\,
        \frac{\G(z \vt t_1,x_1,\xi_1) e^{zv-\zeta u}}{\G(\zeta \vt t_1, x_1, \xi_1) (z-\zeta)}, \\
        &= e^{\mu(v-u)} \ind{v \leq 0}\, \oint \limits_{\Gamma_{-d}} d\zeta \oint \limits_{\Gamma_{D}} dz\,
        \frac{\G(z \vt t_1,x_1,\xi_1) e^{zv-\zeta u}}{\G(\zeta \vt t_1, x_1, \xi_1) (z-\zeta)}. \\
    \end{align*}
    In the second line we deformed the contour $W_{\delta}$ back to the vertical line $\Gamma_{-d}$.

    Note that $||B||_{tr} = (\alpha + \beta) ||f_1||_{L^2} ||g||_{L^2}$. We observe that
    $$ ||g||_{L^2} \leq |\G(\alpha \vt t_1,x_1,\xi_1)| \int_{-\infty}^{0} e^{2 \mu v} < \infty$$
    uniformly in $0 < \alpha < 1$. Also, $f_1(u) = I_{\beta}(u)$ from Lemma \ref{lem:Ibeta}, and so $\sup_{\beta \in (0,1)} ||f_1||_{L^2} < \infty$ by that lemma.
\end{proof}

\begin{prop} \label{prop:J3}
    Suppose $0 < \alpha, \beta < 1$. The kernel $\J_3$ admits the decomposition
    $$ \J_3 = A_3 + (\alpha + \beta) (f_3 \otimes g_3) + E,$$
    where $A_3$ is from Definition \ref{def:A} and
    \begin{align*}
        f_3(u) &= e^{-\mu u} \oint \limits_{\Gamma_{-d_1}} d\zeta\, \frac{e^{-\zeta u}}{\G(\zeta \vt t_1,x_1,\xi_1) (\zeta+\beta)},\\
        g_3(v) &= e^{\mu v} \oint \limits_{\Gamma_{-d_2}} d\omega  \oint \limits_{\Gamma_{D_2}} dw\,
        \frac{\G(\alpha \vt t_1,x_1,\xi_1) \G(w \vt \D t, \D x, \D \xi) e^{\omega v}}{\G(\omega \vt \D t, \D x, \D \xi) (w-\omega)(w-\alpha)}.
    \end{align*}
    In the above, $0 < d_1,d_2 < \beta$ and $\alpha < D_2$. If $\mu > \max \{x_1t_1^{-1/3}, \D x (\D t)^{-1/3}\}$, then there is a constant $C$ such that $\sup_{0<\alpha,\beta < 1} || f_3 \otimes g_3||_{tr} \leq C$ and $||E||_{tr} \to 0$ as $\alpha, \beta \to 0$.
\end{prop}

\begin{proof}
    In the definition of $\J_3$, we can deform the $\zeta$-contour from the vertical line $\Gamma_{-d_1}$ to $W_{\delta}$ for any $0 < \delta < D_1$.
    Next, push the $w$-contour ($\Gamma_{D_2}$) followed by the $z$-contour ($\Gamma_{D_1}$) to the right of $\alpha$ (since $D_1 < D_2 < \alpha)$, and obtain a residue at $z = \alpha$. This gives $\J_3 = A_{\alpha,\beta} + (\alpha+\beta) f_3 \otimes g_3$, where, with $\alpha < D_1 < D_2$,
    \begin{align*}
        A_{\alpha,\beta}(u,v) &= e^{\mu(v-u)} \, \oint \limits_{W_{\delta}} d\zeta \oint \limits_{\Gamma_{-d_2}} d\omega \oint \limits_{\Gamma_{D_1}} dz \oint \limits_{\Gamma_{D_2}} dw\, \\
        & \frac{\G(z \vt t_1,x_1,\xi_1) \G(w \vt \D t, \D x, \D \xi)e^{\omega v-\zeta u}}{\G(\zeta \vt t_1, x_1, \xi_1) \G(\omega \vt \D t, \D x, \D \xi)(z-\zeta)(w-\omega)(z-w)}.
    \end{align*}
    We can take the limit as $\alpha, \beta \to 0$ of $A_{\alpha,\beta}$ in the trace norm to obtain $A_3$, just like in the prior proof.
    We note that $|| f_3 \otimes g_3||_{tr} = ||f_3||_{L^2} ||g_3||_{L^2}$, and, since $f_3 = f_1 = I_{\beta}$, Lemma \ref{lem:Ibeta} gives $\sup_{0 < \beta < 1} ||f_3||_{L^2} < \infty$. Finally, we observe that
    $$ g_3(v) = \G(\alpha \vt t_1, x_1, \xi_1) \oint \limits_{\Gamma_{D_2}} dw\, \frac{\G(w \vt \D t, \D x, \D \xi)}{(w-\alpha)}\, I_w(v),$$
    where $I_w(v)$ is the quantity from Lemma \ref{lem:Iz}. Then, that lemma implies
    $$ |g_3(v)| \leq C e^{(\mu - \D x (\D t)^{-1/3})v} \Ai((\D t)^{-1/3} v)$$
    where
    $$ C = \frac{|\G(\alpha \vt t_1,x_1,\xi_1)|}{D_2 - \alpha} C_{\D t, \D x, \D \xi} \int_{-\infty}^{\infty} ds\, |\G(d_2+\mathbold{i}s \vt \D t, \D x, \D \xi)|.$$
    It follows from the bound above that $\sup_{0 < \alpha < 1} ||g_3||_{L^2} < \infty$.
\end{proof}

\begin{prop} \label{prop:J4}
    Suppose $0 < \alpha, \beta < 1$. The kernel $\J_4$ admits the decomposition
    $$ \J_4 = A_4 + (\alpha + \beta) B_4 + E,$$
    where $A_4$ is from Definition \ref{def:A}. If $\mu > \max \{x_1 t_1^{-1/3}, \D x (\D t)^{-1/3}\}$, then
    there is a constant $C$ such that $\sup_{0<\alpha,\beta < 1} || B_4||_{tr} \leq C$ and $||E||_{tr} \to 0$ as $\alpha, \beta \to 0$.
\end{prop}

\begin{proof}
    The proof is similar to the proof of Proposition \ref{prop:J3}.
    In the definition of $\J_4$, push the $z$-contour to the right of $\alpha$ and obtain a residue at $z = \alpha$.
    Then push the $w$-contour to the right of $\alpha$ as well (no residue here). This gives a decomposition $\J_4 = A_{\alpha,\beta} + (\alpha+\beta) B_4$:
    \begin{align*}
        B_4(u,v) &= e^{\mu(v-u)} \oint \limits_{\Gamma_{-d_1}} d\zeta \oint \limits_{\Gamma_{-d_2}} d\omega \oint \limits_{\Gamma_{D_2}} dw\, \\
        & \frac{\G(\alpha \vt t_1,x_1,\xi_1) \G(w \vt \D t, \D x, \D \xi)}{\G(\zeta \vt t_1,x_1,\xi_1) \G(\omega \vt \D t, \D x, \D \xi)}
        \frac{e^{\omega v-\zeta u}}{(\zeta+\beta) (w-\alpha) (w-\omega)}.
    \end{align*}
    In the above $0 < d_1,d_2 < \beta$ and $D_2 < \alpha$.
    
    In the limit as $\alpha, \beta \to 0$, $A_{\alpha,\beta} \to A_4$ in the trace norm.
    Regarding $B_4$, if we push the $w$-contour to the right of $\alpha$ we obtain a residue at $w = \alpha$.
    This leads to the decomposition
    $$ B_4 = f \otimes g_{4,1} + f \otimes g_{4,2}$$
    where $f = I_{\beta}(u)$ (as in Lemma \ref{lem:Ibeta}) and $g_{4,i}$ are two functions similar to $g_3$ above.
    Arguing as in the proof of Proposition \ref{prop:J3}, we find that $\sup_{0 < \beta < 1} ||f||_{L^2} < \infty$
    as well as $\sup_{0 < \alpha < 1} ||g_{4,i}||_{L^2} < \infty$. The lemma follows.
\end{proof}

\subsection{Invertibility of kernels}
Recall the kernel
$$ A(\theta) = \theta^{\ind{u > 0}}(A_2 - A_1 + A_3) + \theta^{-\ind{u \leq 0}}(A_1 - A_2 + A_4).$$
Recall the functions $f_1$ and $g_1$ from Lemma \ref{prop:J1}. We can decompose $f_1 = I_{\beta} = f_{\beta} + I'_{\beta}$, where
\begin{equation} f_{\beta}(u) = e^{(\beta-\mu)u} \G(-\beta \vt t_1, x_1, \xi_1)^{-1}. \end{equation}
Define the rank one kernel $R(\theta)$ according to
\begin{equation} \label{eqn:Rtheta}
    [R(\theta)](u,v) = \theta^{\ind{u > 0}} f_{\beta}(u) g_1(v).
\end{equation}
The kernel $F(\theta)$ from Theorem \ref{thm:1} has the decomposition:
\begin{equation} \label{eqn:Fdecomp}
    F(\theta) = A(\theta) + E(\theta) - (\alpha + \beta) R(\theta).
\end{equation}
Owing to Propositions \ref{prop:J1}, \ref{prop:J3} and \ref{prop:J4}, we have that as $\alpha, \beta \to 0$,
$$ \max_{|\theta|=r} || F(\theta) - A(\theta)||_{tr} \to 0, \quad \max_{|\theta|=r} ||E(\theta)||_{tr} \to 0.$$
Here $r > 0$ is arbitrary.

The mapping $\theta \to \dt{I + F(\theta)}$ is holomorphic over $\theta \in \C \setminus \{0\}$. It is clearly not identically zero.
Therefore, is has a discrete set of zeroes. So there is a discrete set $Z \subset (1,\infty)$ such that if $r \notin Z$ then
$$ \min_{|\theta|=r} |\dt{I + F(\theta)}| > 0.$$
For such an $r$, $I + F(\theta)$ is invertible for every $|\theta| = r$ and $\max_{|\theta|=r} || (I + F(\theta))^{-1}||_{op} < \infty$.
Indeed, for any trace class kernel $K$ we have $||(I+K)^{-1}||_{op} \leq e^{||K||_{tr}} |\det{(I+K})|^{-1}$.

Since being invertible is an open condition for operators, for $r \notin Z$ and all $\alpha, \beta$ sufficiently small, both $I + A(\theta)$ and $I + A(\theta) + E(\theta)$ are invertible for every $|\theta| = r$. We have thus established the following lemma.

\begin{lem} \label{lem:inversion}
    There is a discrete (and hence countable) set $Z \subset (1,\infty)$ such that if $r \notin Z$ and $\alpha, \beta$ are sufficiently small, then
    the operators $I + F(\theta)$, $I + A(\theta)$ and $I + A(\theta) + E(\theta)$ are all invertible for every $|\theta|=r$ and the operator norm
    of the inverses are bounded over $|\theta| = r$.
\end{lem}

\subsection{Decay estimates}
We need to establish a decay estimate for the kernel $A(\theta)$.
\begin{lem} \label{lem:Adecay}
    If $\mu$ is sufficiently large then there is a constant $C = C_{t_1,x_1,\xi_1, \D t, \D x, \D \xi}$ such that
    $$ |A_i(u,v)| \leq C e^{-u} |\Ai(-t_1^{-1/3}u)| \cdot e^{v} |\Ai((\D t)^{-1/3} v)| \quad i \in \{1,2,3,4\}.$$
    In particular, if $|\theta| = r$ then there is a constant $C'$ such that
    $$ |A(\theta)[u,v]| \leq C' e^{-u} |\Ai(-t_1^{-1/3}u)| \cdot e^{v} |\Ai((\D t)^{-1/3} v)|.$$
\end{lem}

\begin{proof}
    We will prove the decay estimate for $A_3$. The proof for $A_4$ is similar and even simpler for $A_1$ and $A_2$.

    In the definition of $A_3$, we can write $(z-\zeta)^{-1} = \int_{-\infty}^0 d\lambda_1 e^{\lambda_1(z-\zeta)}$
    and $(w-\omega)^{-1} = \int_{-\infty}^0 d\lambda_2 e^{\lambda_1(w-\omega)}$. We then find that
    \begin{align*}
        A_3(u,v) &= e^{\mu(v-u)} \int_{-\infty}^0 d\lambda_1 \int_{-\infty}^0 d\lambda_2 \oint \limits_{\Gamma_{-d_1}} d\zeta \oint \limits_{\Gamma_{-d_2}} d\omega \oint \limits_{\Gamma_{D_1}} dz \oint \limits_{\Gamma_{D_2}} dw\, \\ 
        & \G(\zeta \vt t_1,x_1,\xi_1)^{-1} e^{-\zeta(u+\lambda_1)} \times \G(\omega \vt \D t, \D x,\D \xi)^{-1} e^{\omega(v-\lambda_2)} \times \\
        & \G(z \vt t_1,x_1,\xi_1) \times \G(w \vt \D t, \D x,\D \xi) \times (z-w)^{-1}.
    \end{align*}
    We have that
    \begin{align*}
    \oint \limits_{\Gamma_{-d_1}} d\zeta \frac{e^{-\zeta(u+\lambda_1)}}{\G(\zeta \vt t_1,x_1,\xi_1)} &=
    t_1^{-1/3} e^{-\frac{2}{3}x_1^3-x_1\xi_1 +x_1 t_1^{-1/3}(u+\lambda_1)} \Ai(\xi_1 + x_1^2 - t_1^{-1/3}(u+\lambda_1)),\\
    \oint \limits_{\Gamma_{-d_2}} d\omega \frac{e^{\omega(v-\lambda_2)}}{\G(\omega \vt \D t,\D x,\D \xi)} &=
    (\D t)^{-1/3} e^{-\frac{2}{3}(\D x)^3-(\D x)(\D \xi) +(\D x) (\D t)^{-1/3}(\lambda_2-v)} \times \\
    & \Ai(\D \xi + (\D x)^2 + (\D t)^{-1/3}(v-\lambda_2)).
    \end{align*}
    
    Let
    $$I(\lambda_1,\lambda_2) = \oint \limits_{\Gamma_{D_1}} dz \oint \limits_{\Gamma_{D_2}} dw\,
    \frac{\G(z \vt t_1,x_1,\xi_1) G(w \vt \D t, \D x, \D \xi) e^{\lambda_1 z + \lambda_2 w}}{z-w}.$$
    By writing the above in terms of Airy functions it follows that if $\lambda_1,\lambda_2 \leq 0$, then $|I(\lambda_1,\lambda_2)| \leq C''$ for some
    constant $C''$ that depends on $t_i,x_i,\xi_i$, $i=1,2$.

    Then we see that there is a constant $C$ depending on the same parameters such that
    \begin{align*} |A_3(u,v)| &\leq C \int_{-\infty}^0 d\lambda_1 \int_{-\infty}^0 d\lambda_2 \\
    & e^{(x_1 t_1^{-1/3} - \mu)u+ x_1 t_1^{-1/3} \lambda_1)} |\Ai(\xi_1 + x_1^2 - t_1^{-1/3}(u+\lambda_1))| \times \\
    & e^{(\mu - \D x (\D t)^{-1/3})v+ \D x (\D t)^{-1/3} \lambda_2)} |\Ai(\D \xi + (\D x)^2 + (\D t)^{-1/3}(v-\lambda_1))|.
    \end{align*}
    Due to the rapid decay of the Airy function, the contribution to the integral comes from $\lambda_1, \lambda_2 \approx O(1)$.
    Then by choosing $\mu$ sufficiently large, the bound follows.
\end{proof}

As a corollary to the lemma, we obtain that for any $r > 0$,
\begin{equation} \label{eqn:Athetadet}
    \max_{|\theta|=r} | \dt{I + A(\theta)}| < \infty
\end{equation}
because the Fredholm series expansion converges absolutely. Moreover, we may write each $A_i$ as a product of two Hilbert-Schmidt kernels with entries that obey the bound from Lemma \ref{lem:Adecay}. Therefore, we have that for every $r > 0$,
\begin{equation} \label{eqn:Atracenorm}
    \max_{|\theta|=r} ||A(\theta)||_{tr} < \infty.
\end{equation}

\subsection{Taking the limit}
We begin with a trace calculation.
\begin{lem} \label{lem:tracecalc}
    Let $R(\theta)$ be the kernel from \eqref{eqn:Rtheta}. We have that
    $$ \mathrm{Tr}( R(\theta)) = \frac{\G(\alpha \vt t_1,x_1,\xi_1) \G(\beta \vt t_1, -x_1,\xi_1)}{\alpha + \beta}.$$
    Furthermore
    \begin{equation} \label{eqn:GGlimit}
    \lim_{\alpha,\beta \to 0} \frac{1-\G(\alpha \vt t_1,x_1,\xi_1) \G(\beta \vt t_1, -x_1,\xi_1)}{\alpha + \beta} = t_1^{1/3} \xi_1.
    \end{equation}
\end{lem}

\begin{proof}
    From the definitions of $f_{\beta}$ and $g_1$ we find that (note that $g_1$ is supported on $(-\infty, 0]$)
    \begin{align*}
        \mathrm{Tr}(R(\theta)) &= \int_{-\infty}^{\infty}du\, f_{\beta}(u) g_1(u) \theta^{\ind{u > 0}} \\
        &= \int_{-\infty}^0 e^{(\alpha+\beta)u} \frac{\G(\alpha \vt t_1,x_1,\xi_1)}{\G(-\beta \vt t_1,x_1,\xi_1)} \\
        &= \frac{1}{\alpha + \beta} \frac{\G(\alpha \vt t_1,x_1,\xi_1)}{\G(-\beta \vt t_1,x_1,\xi_1)}\\
        &= \frac{\G(\alpha \vt t_1,x_1,\xi_1) \G(\beta \vt t_1, -x_1,\xi_1)}{\alpha+\beta}.
    \end{align*}

    Observe that
    \begin{align*}
    \G(\alpha \vt t_1,x_1,\xi_1) \G(\beta \vt t_1, -x_1,\xi_1) &= \exp \{\frac{t_1}{3} (\alpha^3+\beta^3) + t_1^{2/3} x_1 (\alpha^2-\beta^2) - t_1^{1/3}\xi_1(\alpha+\beta)\} \\
    &= \exp \{ (\alpha+\beta) (\frac{t_1}{3} (\alpha^2-\alpha \beta + \beta^2) + t_1^{2/3} x_1(\alpha-\beta) - t_1^{1/3}\xi_1)\}
    \end{align*}
    From Taylor expansion of the exponential we deduce that
    $$ \lim_{\alpha,\beta \to 0} \frac{1- \G(\alpha \vt t_1,x_1,\xi_1) \G(\beta \vt t_1, -x_1,\xi_1)}{\alpha+\beta} = t_1^{1/3} \xi_1.$$
\end{proof}

\begin{lem} \label{lem:L2conv}
Let $r > 1$ be such that the conclusion of Lemma \ref{lem:inversion} holds. Define
\begin{align*}
    f_{\beta}(u) &= e^{u(\beta-\mu)} \G(\beta \vt t_1,-x_1,\xi_1) \theta^{\ind{u > 0}} \\
    g_{\alpha}(v) &= e^{v(\mu + \alpha)} \ind{v \leq 0} G(\alpha \vt t_1,x_1,\xi_1).
\end{align*}
Then,
$$ \sup_{|\theta|=r}\, \left | \langle (I + A(\theta))^{-1} A(\theta) f_{\beta}, g_{\alpha} \rangle - \langle (I + A(\theta))^{-1} A(\theta) f_{0}, g_{0}\rangle \right | \to 0$$
as $\alpha,\beta \to 0$. The symbol $\langle \rangle$ denotes the inner product in $L^2(\R)$. The latter inner product is finite and bounded over $|\theta| = r$.
\end{lem}

\begin{proof}
    Assume that $0 \leq \alpha, \beta \leq 1$.
    
    We have the pointwise limit $g_{\alpha}(u) \to g_{0}(u)$ as $\alpha \to 0$. Moreover,
    $$ g_{\alpha}^2(v) \leq C e^{2 \mu v} \ind{v \leq 0}, \quad 0 \leq \alpha \leq 1.$$
    The dominated convergence theorem then implies $|| g_{\alpha}-g_0||_{2} \to 0$ and $\sup_{0\leq \alpha \leq 1} ||g_{\alpha}||_2 < \infty$.
    Here $|| \cdot ||_2$ is the norm in $L^2(\R)$.

    We have the pointwise limit $A(\theta)[u,v]f_{\beta}(v) \to A(\theta)[u,v]f_{0}(v)$ as $\beta \to 0$. From Lemma \ref{lem:Adecay} we have the bound
    $$ |A(\theta)[u,v]f_{\beta}(v)| \leq C e^{-u} \Ai(-t_1^{-1/3}u) e^{v} \Ai((\D t)^{-1/3} v) e^{-\mu u}( e^u \ind{u > 0} + \ind{u \leq 0})$$
    uniformly over $|\theta| = r$ and $0 \leq \beta \leq 1$. The right side above is square integrable over $u,v \in \R$.
    The dominated convergence theorem implies $\sup_{|\theta|=r}|| A(\theta) (f_{\beta}-f_{0})||_2 \to 0$ and $\sup_{|\theta|=r, 0 \leq \beta \leq 1} ||A(\theta) f_{\beta}||_2 < \infty$.

    Now suppose $C$ is a constant such that $||(I+A(\theta))^{-1}||_{op} \leq C$ for all $|\theta|=r$, $||g_{\alpha}||_2 \leq C$ for all $0 \leq \alpha \leq 1$,
    and $||A(\theta)f_{\beta}||_2 \leq C$ for all $0 \leq \beta \leq 1$. We then have,
    \begin{align*}
       & \left | \langle (I + A(\theta))^{-1} A(\theta) f_{\beta}, g_{\alpha} \rangle - \langle (I + A(\theta))^{-1} A(\theta) f_{0}, g_{0}\rangle \right | \\ 
       \leq & \left | \langle (I + A(\theta))^{-1} A(\theta) (f_{\beta}-f_0), g_{\alpha} \rangle \right | + 
       \left | \langle (I + A(\theta))^{-1} A(\theta) f_{\beta}, (g_{\alpha}-g_0) \rangle \right | \\
       &\leq ||(I+A(\theta))^{-1}||_{op} \left ( \langle A(\theta) (f_{\beta}-f_0), g_{\alpha} \rangle + \langle A(\theta) f_{\beta}, (g_{\alpha}-g_0) \rangle\right) \\
       &\leq ||(I+A(\theta))^{-1}||_{op} (||A(\theta) (f_{\beta}-f_0) ||_2 ||g_{\alpha} ||_2 + ||A(\theta) f_{\beta} ||_2 ||g_{\alpha}-g_0 ||_2 )\\
       &\leq C^2 (||A(\theta) (f_{\beta}-f_0) ||_2 + ||g_{\alpha}-g_0 ||_2).
    \end{align*}
    The final quantity tends to zero as $\alpha, \beta \to 0$. We also have $| \langle (I + A(\theta))^{-1} A(\theta) f_{0}, g_{0}\rangle| \leq C^3$.
\end{proof}

Observe that the quantity $T(\theta)$ from Theorem \ref{thm:stat}, \eqref{eqn:Ttheta}, equals
\begin{equation} \label{eqn:Ttheta2}
    T(\theta) = t_1^{1/3} \xi_1 + \langle (I + A(\theta))^{-1} A(\theta) f_{0}, g_{0}\rangle.
\end{equation}
The main lemma is
\begin{lem} \label{lem:limitcalc}
As $\alpha, \beta \to 0$,
    $$\frac{\dt{I+ F(\theta)}}{\alpha + \beta} \to \dt{I + A(\theta)} \cdot T(\theta).$$
    This convergence holds uniformly over $|\theta| = r$ where $r > 1$ is according to Lemma \ref{lem:inversion}.
    For such $r$, the right side in well-defined and bounded over $|\theta|=r$.
\end{lem}

\begin{proof}
    We decompose $F(\theta)$ as in \eqref{eqn:Fdecomp}. Thus,
    $$ I + F(\theta) = I + A(\theta) - (\alpha+\beta) R(\theta) + E(\theta).$$
    Assume $\alpha, \beta$ are small enough and $r$ is such that Lemma \ref{lem:inversion} holds for every $|\theta| = r$.
    In the argument below to only consider $\theta$ lying on the circle $|\theta| = r$.

    Since $R(\theta)$ has rank 1,
    $$ \dt{I + F(\theta)} = \dt{I + A(\theta) + E(\theta)} \left( 1 - (\alpha+\beta) \mathrm{Tr}((I+A(\theta)+E(\theta))^{-1} R(\theta)\right).$$
    Set $L = A(\theta) + E(\theta)$. Since $\max_{|\theta|=r}||E(\theta)||_{tr} \to 0$ as $\alpha, \beta \to 0$, $L \to A(\theta)$ uniformly in the
    trace norm over $\theta$.
    
    We use the identity
    $$ (I+L)^{-1} R(\theta) = (I+L)^{-1}(I+L-L) R(\theta) = R(\theta) - (I+L)^{-1}L R(\theta).$$
    This implies
    $$ \dt{I + F(\theta)} = \dt{I + L} \left( 1 - (\alpha+\beta) \mathrm{Tr}(R(\theta)) + (\alpha+\beta) \mathrm{Tr}((I+L)^{-1} L R(\theta))\right).$$
    By Lemma \ref{lem:tracecalc},
    $$\frac{\dt{I + F(\theta)}}{\alpha+\beta} =
    \left [ \frac{1 - \G(\alpha \vt t_1,x_1,\xi_1)\G(\beta \vt t_1,-x_1,\xi_1)}{\alpha + \beta} + \mathrm{Tr}((I+L)^{-1}LR(\theta))\right ] \dt{I+L}.$$
    We have that $\dt{I+L} \to \dt{I + A(\theta)}$ uniformly in $\theta$ because $||E(\theta)||_{tr} \to 0$ uniformly in $\theta$ as $\alpha, \beta \to 0$.
    Similarly, $\mathrm{Tr}((I+L)^{-1}LR(\theta)) - \mathrm{Tr}((I+A(\theta))^{-1} A(\theta)R(\theta)) \to 0$ uniformly in $\theta$.
    Finally, we note that
    $$\mathrm{Tr}((I+A(\theta))^{-1} A(\theta)R(\theta)) = \langle (I + A(\theta))^{-1} A(\theta) f_{\beta}, g_{\alpha} \rangle \to
    \langle (I + A(\theta))^{-1} A(\theta) f_{0}, g_{0}\rangle$$
    uniformly in $\theta$ by Lemma \ref{lem:L2conv}. So by \eqref{eqn:GGlimit}, the lemma follows.   
\end{proof}
\subsection{Completing the proof}
By Theorem \ref{thm:1} and Lemma \ref{lem:shift} we find that
$$\pr{\mathcal{L}(h_0; x_i,t_i) \leq \xi_i, i=1,2} = \lim_{\alpha,\beta \to 0}\, (1 + \frac{1}{\alpha+\beta}(\partial_{\xi_1} + \partial_{\xi_2}))
\frac{1}{2 \pi \mathbold{i}} \oint \limits_{|\theta|=r} d\theta\, \frac{\dt{I + F(\theta)}}{\theta-1}.$$

Suppose $r$ is chosen according to Lemma \ref{lem:inversion}. The right side above boils down the evaluating the limit as $\alpha,\beta \to 0$ of
$$\frac{1}{2 \pi \mathbold{i}} \oint \limits_{|\theta|=r} \frac{d\theta}{\theta-1} (\partial_{\xi_1} + \partial_{\xi_2}) \frac{\dt{I + F(\theta)}}{\alpha+\beta}.$$
By Lemma \ref{lem:limitcalc}, the limit of $\dt{I+F(\mathbold{\theta})}/(\alpha+\beta)$ equals $\dt{I + A(\theta)} \cdot T(\theta)$, as required.

\end{document}